\documentclass[10pt]{article}

\usepackage{amssymb,amsmath}
\usepackage{amsthm}
\usepackage{esint}
\usepackage{xcolor}
\usepackage{mathrsfs}
\newcommand{\vast}{\bBigg@{4}}
\newcommand{\Vast}{\bBigg@{5}}
\makeatother
\newcommand{\eps}{\varepsilon}

\newcommand{\N}{{\mathbb{N}}}

\newcommand{\R}{\mathbb{R}}
\newcommand{\p}{\partial}

\newcommand{\Dsh}{(-\Delta)^{s/2}}

\def\H{{\mathcal H}}

\newcommand{\E}{\mathscr{E}}
\newcommand{\J}{\mathscr{J}}
\newcommand{\G}{\mathscr{G}}
\newcommand{\F}{\mathscr{F}}

\def\H{\mathcal H}

\def\div{{\mathrm {div}}}

\theoremstyle{plain}

\numberwithin{equation}{section}
\newtheorem{theorem}{Theorem}[section]
\newtheorem{lemma}[theorem]{Lemma}

\newtheorem{remark}[theorem]{Remark}
\newtheorem{proposition}[theorem]{Proposition}
\newtheorem{corollary}[theorem]{Corollary}

\hbadness=\maxdimen
\title{
An isoperimetric problem with a competing nonlocal  singular term}

\author{Antoine Mellet\thanks{mellet@umd.edu. Partially supported by NSF Grant DMS-2009236.}  and Yijing Wu\thanks{yijingwu@umd.edu}}

\date{University of Maryland \\ Department of Mathematics \\ College Park, MD 20742  USA}

\begin{document}

%\begin{frontmatter}

\maketitle

\begin{abstract}
In this paper, we investigate the minimization of a functional in which the usual perimeter is competing with a nonlocal singular term comparable (but not necessarily equal to) a fractional perimeter. The motivation for this problem is a cell motility model introduced in some previous work by the first author.
We establish several facts about global minimizers with a volume  constraint.
In particular we prove that minimizers exist and are radially symmetric for small mass, while
minimizers cannot be radially symmetric for large mass. For large mass, we prove that the minimizing sequences either split into smaller sets that drift to infinity or develop fingers of a prescribed width.
Finally, we connect these two alternatives to a related minimization problem for the optimal constant in a classical interpolation inequality (a Gagliardo-Nirenberg type inequality for fractional perimeter).
\end{abstract}

\section{Introduction}
\subsection{Setting of the problem}
Given a monotone decreasing radially symmetric kernel $K(x)$ satisfying
\begin{equation}\label{eq:assK}
0\leq K(x) \leq \frac{1}{|x|^{n+s}} \qquad \forall x\in \R^n, \mbox{ for some $s\in(0,1)$}
\end{equation}
we consider the volume constrained minimization of the following energy functional:
\begin{equation}\label{eq:Ja}
\J_{\alpha}(E):= \int_{\R^n} |D\chi_E| -\alpha  s(1-s) \int_{\R^n}{\int_{\R^n} {K(x-y)|\chi_E(x)-\chi_E(y)|dx}dy},  \end{equation}
where $\chi_E$ is the characteristic function of the set $E\subset \R^n$ and $\alpha$ is a positive parameter.
%We will also assume that $K$ is monotone decreasing (that is $K(x)=k(|x|)$ with $k:(0,\infty)\to[0,\infty)$ non-increasing function satisfying 
This functional involves the classical perimeter
\begin{align*}
P(E) & = \int_{\R^n} |D\chi_E| : = \sup\left\{\int_{\R^n} \chi_E \, \mathrm{div}\, g\, dx \,;\, g\in [C^1_c(\R^n)]^n, \; |g|\leq 1   \right\}
\end{align*}
and the nonlocal perimeter
\begin{align*}
P_K(E) & =  \int_{\R^n}{\int_{\R^n} {K(x-y)|\chi_E(x)-\chi_E(y)|dx}dy}= 2 \int_{\R^n} \int_{\R^n} K(x-y) \chi_E(x) \chi_{E^c} (y) dx dy .
\end{align*}
Condition \eqref{eq:assK} implies that 
$$ 0\leq P_K(E) \leq P_s(E)$$
where 
$$P_s(E): = \int_{\R^n} \int_{\R^n} \frac{|\chi_E(x)-\chi_E(y)|}{|x-y|^{n+s}} dx dy$$
is the usual $s$-perimeter (up to a constant), which has received a lot of attention \cite{cafroq,cafval,figV,ambrosio,val13} and appears in many applications.
The functional \eqref{eq:Ja} models phenomena in which the classical perimeter competes with a nonlocal interaction term. While the perimeter tries to aggregate the set $E$ into a ball, the nonlocal term has the opposite effect. This competition leads to a non trivial problem in which even the existence of minimizers is not obvious.
This problem is reminiscent of several recent works, including many results on the classical Gamow's liquid drop model of an atomic nucleus 
\cite{Muratov16,Muratov17,Muratov18,muratov-general,Muratov-planar}.
More recently, some results were obtained in \cite{DiCastro} in a framework similar to ours, when  $P_K=P_s$.

We will see that for small volume (or small $\alpha$), the perimeter dominates the nonlocal effect and that the minimizers exist and are balls (see Theorems \ref{thm:existence} and \ref{thm:ball}), while for larger volume, the ball cannot be the global minimizer (see Theorem \ref{thm:nonball}). 
Our results will be uniform in $K$ under assumption \eqref{eq:assK} and consistent with the natural scaling of the problem when $P_K=P_s$. This is important for the application we have in mind, which is discussed in Section \ref{sec:cell} below, in which the kernel $K$ itself depends on a small parameter $\eps$.

For large volume, the situation is less clear, and we show that two things might happen: Either the minimizers still exists but are no longer balls. Instead the set $E$ develops finger of width of the order of $\alpha^{-\frac{1}{1-s}}$. 
Or the minimizer no longer exists because the minimizing sequence splits into smaller sets. 
Indeed, we recall the following important properties of nonlocal perimeters: For disjoint sets $A$, $B$, we have
\begin{equation}\label{eq:PKAB}
P_K(A\cup B) = P_K(A) + P_K(B)  - 4\int_{\R^n} \int_{\R^n} K(x-y) \chi_A(x) \chi_B(y)\, dx\,dy . 
\end{equation}
If the sets $A$ and $B$ do not "touch" each others (so that $P(A\cup B)=P(A)+P(B)$), 
it follows that
$$
\J_\alpha(A\cup B) = \J_\alpha(A) + \J_\alpha(B) + \alpha s(1-s) 4\int_{\R^n} \int_{\R^n} K(x-y) \chi_A(x) \chi_B(y)\, dx\,dy$$
and so the energy decreases if the sets $A$ and $B$ are translated away from each others. From a mathematical point of view, this means that minimizers do not exist if the minimizing sequence splits into smaller sets.  However,
from the point of view of many applications, the interaction energy can be neglected if the sets $A$ and $B$ are far enough from each others, and it is interesting to study the structure of the minimizers when this term is neglected. This is the object of our Theorem \ref{thm:G} which shows that if the connected components do not interact with each other, then a global minimizer always exist. Furthermore, we can get a lower bound on the volume of the each set, and thus on the number of disconnected component.

We note that numerical computations (see \cite{Glasner}) show the existence of local (bounded) minimizers with intricate patterns of fingers, although these might not be the global minimizers. 

While we do not determine whether minimizers exist (and form fingers) or not (and split) for large volume, we show (see Proposition \ref{prop:JGN}) that when $P_K=P_s$ this question is directly related to the optimal constant in the following classical interpolation inequality for sets of finite perimeter, which plays an important role in our paper:
\begin{proposition}\label{prop:GN}
Denote by $\omega_n$ the volume of the ball $B_1$ in $\R^n$. Then for all $s\in(0,1)$, we have
\begin{equation}\label{eq:GN}
P_K(E)\leq P_s(E)\leq \frac{2^{1-s} n \omega_n}{s(1-s)} P(E)^s|E|^{1-s} \mbox{ for all sets of finite perimeter $E\subset \R^n$.}
\end{equation}
We denote by $\mu_{n,s}\leq 2^{1-s} n \omega_n$ the optimal constant for a given $s$, that is
\begin{equation}\label{eq:opt} 
\mu_{n,s}^{-1} = \inf  \frac{P(E)^s|E|^{1-s}}{s(1-s)P_s(E)} = \inf_{|E|=1}  \frac{P(E)^s}{s(1-s) P_s(E)} .
\end{equation}
%When $K(x)=\frac{1}{|x|^{n+s}}$ (that is $P_K=P_s $), we denote by $\mu_{n,s}$ the optimal constant in \eqref{eq:GN}.
\end{proposition}
The proof of this proposition is classical and presented in Appendix \ref{app:GN} for the sake of completeness. 
To the authors' knowledge, it is an open question whether the optimal constant in \eqref{eq:GN} is reached (i.e. there is equality for a set of finite perimeter) or not (minimizing sequences for \eqref{eq:opt} have unbounded perimeters).

Finally, note that we include the constant $s(1-s)$ in front of $P_K$ in \eqref{eq:Ja} so that we can easily make sense of our result when $s\to 0^+$ and $s\to 1^-$. Indeed, it is known (see for instance \cite{Dipierro13,ambrosio}) that 
$$
\lim_{s\to 1^-} (1-s)P_s(E) = 2\omega_{n-1} P(E), \qquad
\lim_{s\to 0^-} s P_s(E) = 2n\omega_n |E|
$$
so volume constraint minimizers of $\J_\alpha$ always exist and are balls when $s\to 0$ and exist (and are balls) if $\alpha$ is small enough when $s\to1$.

\medskip

In Section \ref{sec:cell} below, we will briefly motivate the problem by relating it to a more classical minimization problem. Our main results are presented and discussed in Section \ref{sec:main} and the rest of the paper is devoted to their proofs.

\subsection{Motivation and related functional}\label{sec:cell}
To motivate this work, we consider the functional
$$
J (E) = \sigma P(E) + \beta V_K(E)$$
where $V_K$ is a nonlocal repulsive interaction energy defined by
\[
V_K(E)  =\int_{\R^n}{  \int_{\R^n} {K(x-y) \chi_{E} (x) \chi_E(y)dx}dy}.
\]
This energy functional has been intensively studied with the kernel $K(x)=\frac{1}{|x|^{n-a}}$ for some $a\in(0,n)$. The case $a=n-1$ and $n=3$ in particular corresponds to the classical Gamow's liquid drop model of an atomic nucleus \cite{Muratov16,Muratov17,Muratov18,muratov-general,Muratov-planar}.
We note that the two terms in $J$ have opposite effects: The perimeter tries to keep the mass together in a ball, while the repulsion potential $V_K$ has the tendency to spread the mass around. 
In \cite{figalli}, a detailed analysis of the related functional $P_s + \alpha V_K$ is done.

The motivation for our study is very different and stems from some model for cell motility recently introduced in \cite{CMM}. In that framework, the kernel  $K$ depends on a small parameter $\eps\ll1$ and is the solution of
\[
K_{\epsilon}+\eps^s\Dsh K_{\eps}=\delta(x).
\]
We note in particular that we can write 
$K_{\eps}(x)=\eps^{-n}K(\frac{x}{\eps})$, with $K$ solution of
%for some $K(x)$ satisfying $\int_{\R^n}{K(x)dx}=1$ and 
\begin{equation}\label{eq:K}
K+\Dsh K =\delta(x).
\end{equation}
%In particular, we have  $\hat{K}(\xi)=\frac{1}{1+|\xi|^s}$.
For $s\in(0,2)$, we have
\begin{lemma}\label{lem:K}
Assume that $s\in(0,2)$. Then the kernel $K$, solution of \eqref{eq:K}, satisfies $K(z)= k(|z|)$ with $k:[0,\infty)\to [0,\infty)$, $\int_{\R^n} K(x)\, dx=1$ and
\begin{equation}\label{eq:Ks}
k(r)\sim \frac{k_0}{r^{n+s}} \mbox{ as } r\to\infty.
\end{equation}
\end{lemma}

Note also that $K(x)\sim \frac{k_1}{|x|^{n-s}}$ when $|x|\to0$, and the energy $V_{K_\eps}$ is comparable to the Reisz potential when $\eps=\mathcal O(1)$, but we are interested here in the regime where $\eps\ll 1$, $\beta\gg1$. 

\medskip

The functional  $V_K$  is related to nonlocal perimeters in the following way:
\begin{align}
V_K(E) & = \int_{\R^n} \chi_E K * \chi_E \, dx\nonumber \\
& = \int_{\R^n}(1-\chi _{\R^n\setminus E}) K * \chi_E \, dx\nonumber \\
& = |E| - \int_{\R^n} \chi_{\R^n\setminus E} K * \chi_E\, dx \nonumber \\
& = |E| - \frac 1 2  P_{K} (E).\label{eq:VK}
\end{align}
Since we are interested in the minimization of $J$ under the volume constraint $|E|=m$, it is equivalent to the minimization of 
\begin{equation}\label{eq:Eeps}
\bar J(E) = P(E) -\frac{\beta}{\sigma} P_{K_\eps}(E)
\end{equation}
with the same constraint.
Finally, when $s\in(0,1)$, the interesting regime corresponds to $\eps\ll1$ and $\frac{\beta}{\sigma} \sim \eps^{-s}$ (this is the regime in which the two terms in $\bar J$ have the same order in $\eps$). If we set $\alpha = \frac{\beta k_0}{\sigma s(1-s)} \eps^s$, we are led to investigate the functional 
$$ \bar J_{\alpha,\eps}(E) =P(E) -\alpha s(1-s) P_{\bar K_\eps}(E), \quad \bar K_\eps (x)= \eps^{-s} K_\eps(x) = \eps^{-(n+s)}K\left(\frac{x}{\eps}\right).$$

Importantly, the kernel $\bar K_\eps$ satisfy the upper bound \eqref{eq:assK} uniformly with respect to $\eps$ (up to a constant), with equality in the limit $\eps\to0$.
The results derived in this paper therefore apply to the functional  $\bar J_{\alpha,\eps}$ uniformly in $\eps$.

\begin{remark}
We note that we only work with $s\in(0,1)$ in this paper, even though Lemma \ref{lem:K} holds for $s\in (0,2)$.
Indeed, when $s>1$, 
the interesting regime corresponds to $\eps\ll1$ and $\frac{\beta}{\sigma} \sim \eps^{-1}$:
If we set $\alpha = \frac{\beta}{\sigma } \eps $, we get the functional 
$$ \bar J_{\alpha,\eps}(E) =P(E) -\alpha   P_{\bar K_\eps}(E), \quad \bar K_\eps (x)= \eps^{-1} K_\eps(x) = \eps^{-(n+1)}K\left(\frac{x}{\eps}\right).$$
and a classical result (see \cite{BBM,Davila}) implies that 
$$ P_{\bar K_\eps}(E) \to \sigma_0 P(E)$$
for some constant $\sigma_0$. So the analysis in that case is somewhat different and in the limit, the behavior depends on whether $\alpha \sigma_0<1$ (in which case the ball is the unique minimizer) or $\alpha \sigma_0>1$ (in which case minimizers do not exist).
%The case $\eps=1$ is critical, but a similar result holds if we take $\alpha = \frac{\beta k_0}{\sigma  } \eps |\ln\eps| $.
\end{remark}
The $\Gamma$-convergence of $P_{\bar K_\eps}$ to $P_s$ when $s\leq 1$ is the object of a forthcoming paper.

\section{Main results}\label{sec:main}
The paper is thus devoted to the following minimization problem in $\R^n$ (expect for Theorem \ref{thm:ball}, all of our results hold for any $n\geq 2$):
\begin{equation}\label{eq:min}
 \inf\big\{      \J_{\alpha}(F)\,;\, |F|=m\big\} ,
\end{equation}
where we recall that
$$\J_{\alpha}(F)= P(F)-\alpha s(1-s) P_K(F).$$
This problem has two parameters, $\alpha $ and $m$.
When we have equality in \eqref{eq:assK}, that is when $P_K = P_s$, these parameters are related by a natural scaling. Indeed, if we define
$$ \F_{s,\alpha}(F) = P(F)-\alpha s(1-s) P_s(F),$$
then $E$ is a minimizer of  $\F_{s,\alpha}$ with $|E|=m$ if and only if $\tilde E=\frac{1}{m^{1/n}}E$ is a minimizer of  $\F_{s,\gamma}$ with $|\tilde E|=1$ where 
\begin{equation}\label{eq:scaling}
\gamma=\alpha m^ {\frac{1-s}{n}}.
\end{equation}
Even though this is no longer valid when we only have the inequality  \eqref{eq:assK}, our results below will be consistent with this scaling.
\medskip

Next, we note that the existence of minimizing sequence, with bounded perimeter, is easy to establish. Indeed,
using \eqref{eq:GN} and Young's inequality, we find
\begin{align*}
s(1-s)\alpha P_K(E) & \leq s(1-s)\alpha P_s(E) \nonumber \\
& \leq 2^{1-s}n\omega _n \alpha P(E)^s|E|^{1-s}  \nonumber \\
& \leq \frac{1}{2} P(E) +  \big( (1-s)^{1-s} s^s n\omega_n \alpha \big)^{\frac{1}{1-s}} |E| 
\end{align*}
and so
\begin{equation}\label{eq:KP}
s(1-s)\alpha P_K(E) \leq \frac{1}{2} P(E) +  \big( n\omega_n \alpha \big)^{\frac{1}{1-s}} |E| 
\end{equation}
(since $1/2\leq  (1-s)^{1-s} s^s < 1$ for all $s\in (0,1)$).
%, we write
%\begin{equation}
%s(1-s)\alpha P_K(E)  \leq \frac{1}{2} P(E) + \frac{1}{2} \big( 2 c_n \alpha \big)^{\frac{1}{1-s}} |E|
%\end{equation}
In particular, we have 
\begin{equation}\label{eq:Jlower} 
\J_\alpha(E) \geq \frac{1}{2} P(E) - \big( n\omega_n \alpha \big)^{\frac{1}{1-s}}  |E| \qquad \mbox{ for all $E\subset \R^n$}
\end{equation}
which  guarantees the existence of a minimizing sequence $\{F_k\}_{k\in\N}$ with volume constraint $|F_k|=m$ and satisfying $P(F_k)\leq C$.
The difficulty in proving the existence of a minimizer for \eqref{eq:min} is that the sets $F_k$ might not be bounded in $\R^n$.
In particular, $F_k$ might split into two (or more) connected components moving away from each others (since translating the components away from each others does not change the perimeter but increases the  nonlocal perimeter, and thus decreases the functional $\J_\alpha$). 
Our first result shows that this does not happen when $\alpha$ (or equivalently $m$) is small enough. 
More precisely, we show the existence of a bounded minimizer for $\J_\alpha$  when the $\alpha\, m^{\frac{1-s}{n}}$ is small enough (which is consistent with the scaling \eqref{eq:scaling}):
%The first result establishes the existence of a bounded minimizer for $\J_\alpha$  when the $\alpha\, m^{\frac{1-s}{n}}$ is small enough. 
We note that a similar result is proved in \cite{DiCastro} for the functional $\F_{s,\alpha}$ (that is when $P_K=P_s$).
\begin{theorem}\label{thm:existence}
Let $K(x)$ be a radially symmetric, non-increasing function satisfying \eqref{eq:assK} for some $s\in(0,1)$.
There exists $\gamma_0>0$ depending only on $n$  such that if
$$
\alpha\, m^{\frac{1-s}{n}}\leq \gamma_0
$$
then the problem \eqref{eq:min} has a bounded minimizer.
%functional $\J_{\alpha} $ has a bounded minimizer $E$ under the constraint $|E|=m$.
\end{theorem}
When $\alpha=0$, the only minimizer of $\J_0=P$ is the ball, and we next show that when $\alpha\, m^{\frac{1-s}{n}}$ is small enough, the unique bounded minimizer of \eqref{eq:min} is also the ball of volume $m$ (again, we note that a similar result is proved in \cite{DiCastro} for the functional $\F_{s,\alpha}$):
\begin{theorem}\label{thm:ball}
Assume $2\leq n\leq 7$ and let $K(x)$ be a radially symmetric, non-increasing function satisfying \eqref{eq:assK} for some $s\in(0,1)$.
There exists $\gamma_1>0$ depending on $n$ and $s$ such that if
$$
\alpha\, m^{\frac{1-s}{n}}\leq \gamma_1
$$
then the
unique (up to translation) bounded minimizer $E$ 
of the problem \eqref{eq:min} is the ball of volume $m$.
\end{theorem}

\medskip

The next step is to determine what happens when $ \alpha\, m^{\frac{1-s}{n}}$ is large. 
In view of our previous results, there are two questions: Does the functional $\J_\alpha$ still have a minimizer when  $ \alpha\, m^{\frac{1-s}{n}}$ is large? and if so, is this minimizer still a ball?
We start with the second part of this question, which is much simpler to answer.
First we prove that when $ \alpha\, m^{\frac{1-s}{n}}$ is large, the global minimizer (if it exists) cannot be the ball, and in fact that this global minimizer must (in some sense) develop fingers. 
We note that such a result cannot hold in the general framework considered so far, since we can take $K=0$ (in which case the ball is the unique minimizer for all $\alpha$). We thus assume that $K$ satisfy the lower bound:
\begin{equation}\label{eq:assK2} 
K(x)\geq \frac{c}{|x|^{n+s}} \qquad \mbox{ when } |x|\geq \rho_0
\end{equation}
for some $c\geq 0$.
\begin{theorem}\label{thm:nonball}
Let $K(x)$ be a radially symmetric, non-increasing function satisfying \eqref{eq:assK} and \eqref{eq:assK2} for some $s\in(0,1)$, and let $E$ be a minimizer of \eqref{eq:min}. 
There exists a constant $c>0$ depending only on $n$ such that if 
$$ \rho\geq \max\{4\rho_0, (cs(1-s) \alpha)^\frac{-1}{1-s}\}$$
then
$$
|  B_\rho(x)\setminus E | >0 \quad \mbox{ for all $x\in\R^n$.}
$$
In particular there exists $\gamma_2>0$ depending only on $n$, $s$ such that if
$$
\alpha\, m^{\frac{1-s}{n}} \geq \max \{\gamma_2, c \alpha \rho_0^{1-s}\} 
$$
then the ball of volume $m$ cannot be a global minimizer of \eqref{eq:min}. 
\end{theorem}
The first part of the theorem says that the largest ball that can fit inside the set $E$ has radius $c \max\{\rho_0,\alpha^\frac{-1}{1-s}\}$ for some $c$ depending only on $n$ and $s$ (note that when $P_K=P_s$, we can take $\rho_0=0$). 
This shows that when $m$ increases with $\alpha$ fixed, minimizers, if they exists, have a maximum width determined by $\rho_0$ and $\alpha$.
Note that the non-degeneracy estimate \eqref{eq:ndg} suggests that this is optimal and that the fingers have width comparable to $\alpha^\frac{-1}{1-s}$.

\medskip

From the point of view of many applications (in particular the cell motility applications discussed in this paper), global minimizers of $\J_\alpha$ are not necessarily the most relevant objects one would like to study. In particular, the stationary solutions obtained by gradually increasing the value of $\alpha$ (or obtained as the long time limit for some related evolution problem) might be a stable critical point (or a local minimizer) but not a global energy minimizer.
In the theorem below, 
we characterize the range of values of $\alpha\, m^{\frac{1-s}{n}}$ for which the ball is a local vs. global minimizer for the functional
$$ \F_{s,\alpha}(F) = P(F)-\alpha s(1-s)P_s(F),$$
(so when $P_K=P_s$).
More precisely, we prove:
%Finally, we further characterize the role of the ball.
%We note that $E$ is a minimizer $\F_{s,\alpha}$ with the constraint $|E|=m$, if and only if $E_m=\frac{1}{m^{1/n}}E$ is a minimizer of energy $\F_{s,\beta}$ under the constraint $|F|=1$ with 
%$\beta=\alpha m ^{\frac{1-s}{n}}$.
%So the following theorem identifies the values of $\alpha$ and $m$ for which the  ball of volume $m$ is a global or local minimizer of $\F_{s,\alpha}$:
\begin{theorem}\label{thm:beta1beta2}
There exists $\gamma_1^*,\gamma_2^*>0$ depending only on $n$ and $s$, and satisfying $\gamma_1^*<\gamma_2^*$ such that 
\begin{itemize}
	\item[(i)]
	The ball of volume $m$ is a global minimizer of $\F_{s,\alpha}$ with $|E|=m$ if and only if $\alpha\, m^{\frac{1-s}{n}} \leq\gamma^*_1$.
	\item[(ii)] When $\alpha\, m^{\frac{1-s}{n}} <\gamma_2^*$, the ball of volume $m$ is a local minimizer of $\F_{s,\alpha}$ with $|E|=m$, and when $\alpha\, m^{\frac{1-s}{n}} >\gamma_2^*$, the ball is not a local minimizer. 
\end{itemize}
\end{theorem}
The fact that there is a non trivial range of values of $\alpha\, m^{\frac{1-s}{n}} $ for which the ball is a local but not a global minimizer comes from the following explicit expression for $\gamma_2^*$:
$$\gamma_2^*=\frac{n+1}{s(n+s)}\frac{P(B_1)|B_1|^{\frac{1-s}{n}} }{s(1-s)P_s(B_1)}$$
(which results from a careful computation of the second variation of the energy)
and the bound
$$ \gamma^*_1 \leq \frac{2^{\frac{1}{n}}-1}{2^{\frac{s}{n}}-1}\frac{P(B_1)|B_1|^{\frac{1-s}{n}}}{s(1-s)P_s(B_1)} $$
(which is obtained by comparing the energy of $B_1$ with the energy of the set made of two balls $B_{\frac{1}{2^{1/n}}}$ far away from each other).

The result implies that a bifurcation phenomena takes place when $\gamma=\gamma_2^*$.
While we do not pursue this in this paper, it is indeed possible to show, by a Crandall-Rabinovitch type argument, that a branch of non-radially symmetric stationary solutions appear  for $\gamma=\gamma_2^*$. This branch of solution has a lower energy than the ball and is thus important in application, even if the fact that $\gamma_2^*>\gamma_1^*$ implies that this branch is not a global minimizer for $\gamma$ close to $\gamma_2^*$ (two balls of half the size far away from each other have a lower energy).

\medskip

Finally, we come back to our first question: the existence (or non existence) of  global minimizer  
 when  $ \alpha\, m^{\frac{1-s}{n}}$ is large.
First, we recall that \eqref{eq:Jlower} implies that any minimizing sequence $\{F_k\}_{k\in \N}$ 
satisfies 
$P(F_k) \leq C$
 for some constant $C(m,\alpha)$. In particular $\chi_{F_k}$ is bounded in $BV(\R^n)$ and the obstacle to the existence of a minimizer is the issue of decomposability:
 
A set of finite perimeter $F$ is decomposable if there exists a partition $F=F_1\cup F_2$ with 
$|F_i|\neq 0$ and $P(F)= P(F_1)+P(F_2)$. Otherwise the set is indecomposable (which is the measure theoretic equivalent of the notion of a connected set).
In Proposition \ref{lem:indecomposable} we will prove that any global minimizer of $\J_\alpha$ must be bounded and indecomposable since otherwise we can decrease the energy by moving two components away from each others (see \eqref{eq:PKAB}).

%. Indeed, this keeps the perimeter constant while increasing the $K-$perimeter since we have:
%$$
%P_K(F_1\cup F_2)  = P_K(F_1) + P_K(F_2) - 4 \int_{F_1} \int_{F_2} K(x-y)\, dx \, dy \qquad \mbox{ if $ |F_1\cup F_2|=0$.}
%$$

In many applications, splitting of the minimizing sequence into smaller sets can be an interesting feature and once the components are far enough from each others we can typically neglect the interactions and treat the resulting set as a minimizer of the problem. Mathematically, this can be done by assuming that $K(x)=0$ for $|x|$ large enough, or by
replacing the original minimization problem with the following:
Given a collection $\mathbb E = \{E_i\}_{i\in I}$ of sets of finite perimeter, let
$$ \G_\alpha(\mathbb E)  = \sum_{i\in I} \J_\alpha (E_i), \qquad (\mbox{ and } \G_{s,\alpha}(\mathbb E)  = \sum_{i\in I} \F_{s,\alpha} (E_i))$$
and $|\mathbb E| = \sum_{i\in I}|E_i|$. We then consider the minimization problem:
\begin{equation}\label{eq:minG}
\inf \left\{ \G_\alpha(\mathbb E) \, ;\, |\mathbb E| = m\right\} .
\end{equation}
For that problem we can prove that a minimizer exists for all $m$ and all $\alpha$.
More precisely, we have the following result:
\begin{theorem}\label{thm:G}
Let $K(x)$ be a radially symmetric, non-increasing function satisfying \eqref{eq:assK}. 
Then the minimization problem \eqref{eq:minG} has a global minimizer $\mathbb E = \{E_i\}_{i\in I} $ for all $m>0$ and $\alpha\geq 0$. Furthermore, there exists $C$ depending only on $s$ and $n$ such that any global minimizer $\mathbb E$ has at most $C \alpha^{\frac{n}{1-s}} m$ components with positive Lebesgue measure, each component has diameter at most 
$$\mathrm{diam}(E_i)\leq  C \max\{1,\alpha^{\frac{n}{1-s}} |E_i|\} |E_i|^{1/n},$$
 and satisfies
$$ P(E_i) \leq C\max\left\{  |E_i|^{\frac{n-1}{n}}, \alpha^{\frac{1}{1-s}} |E_i| \right\}$$
for some constant $C$ depending only on $n$ and $s$.

\item The minimization problem \eqref{eq:min} has a global minimizer if and only if the minimization problem \eqref{eq:minG} has at least one global minimizer with only one non-empty component.
\end{theorem} 
The question of whether a global minimizer of $\J_\alpha$ exists when $\alpha\, m^{\frac{1-s}{n}}$ can thus be reframed as a question on the number of non trivial components of the minimizers of $\G_\alpha$.
\medskip

When $P_K=P_s$,  we can prove the following proposition which relates our minimization problem 
to the interpolation inequality \eqref{eq:GN}: 
\begin{proposition}\label{prop:as}
Let $\mu_{n,s}\geq 0$ be the optimal constant in the interpolation equality \eqref{eq:GN}, defined by 
\begin{equation}\label{eq:opt2}
\mu_{n,s}^{-1} = \inf_{|E|=1}  \frac{P(E)^s}{s(1-s) P_s(E)} .
\end{equation}
Then for all $\alpha>0$, we have:
\[
\inf_{|F|=m} \F_{s,\alpha}(F) = \inf_{|\mathbb F|=m} \G_{s,\alpha}(\mathbb F)= -(1-s)s^{\frac{s}{1-s}} (\alpha\mu_{n,s})^{\frac{1}{1-s}}m+o(m)\qquad  \mbox{when $m\to \infty$.}
\]
\end{proposition}
The proof of this proposition makes use of sets that are almost minimizers for  \eqref{eq:opt2} and it allows us to prove the following result:
\begin{proposition}\label{prop:JGN}
\item[(i)] If there exists a set $E_0$ with $|E_0|=1$ and $P(E_0)<\infty$ for which there is equality in \eqref{eq:opt2}, then
for all $k\in \N$, 
%$$m\in \left\{ k \left(\frac{P(E_0)}{sP_s(E_0)}\right)^{\frac{n}{1-s}}\, ;\,  k\in \N\right\}$$
the set with $k$ components 
$$ \mathbb E_k = \left\{ \left(\frac{m_k}{k}\right)^{\frac{1}{n}}E_0, \dots ,\left(\frac{m_k}{k}\right)^{\frac{1}{n}}E_0 \right\}$$
is a minimizer of $\G_{s,\alpha}$ with mass constraint
$$|\mathbb E_k|=m_k=k \alpha^{-\frac{n}{1-s}} \left(\frac{P(E_0)}{sP_s(E_0)}\right)^{\frac{n}{1-s}}= k \frac{{P(E_0)}^n}{(s\alpha \mu_{n,s})^\frac{n}{1-s}}.
$$ 
\item[(ii)] If there exists a sequence $\{m_k\}_{k=1}^{\infty}$ with $m_k \to \infty$ such that $E_{k}$ are bounded minimizers of $\F_{s,\alpha}$ with volume constraint $|E|=m_k$, then the sets $F_k=m_k^{-\frac{1}{n}}E_k$  is a minimizing sequence for  \eqref{eq:opt2} with
$$   P(F_k)\geq Cm_k^{\frac{1}{n}}\to \infty \mbox{  as $k \to \infty$.}$$ 
\end{proposition}
%a minimizing sequence for $\F_{s,\alpha}$ with volume constraint $|E|=m_k=k \left(\frac{P(E_0)}{sP_s(E_0)}\right)^{\frac{n}{1-s}}$ can be constructed by taking 
%$k$ copies  of the set $(\frac{m_k}{k})^{1/n}E_0$ and sending them to infinity in different directions.
This proposition suggests that the existence of bounded minimizer for $\F_{s,\alpha}$ is directly related to whether minimizing sequences \eqref{eq:opt2}
have bounded perimeter or not. 

In particular, if equality is achieved in \eqref{eq:opt2}, then the first part of Proposition \ref{prop:JGN} suggests that for large enough $m$ minimizing sequences for $\F_{s,\alpha}$ have several connected components. In that case global minimizers of $\F_{s,\alpha}$ do not exist, although from the point of view of application, since we can often neglect the interaction between different connected  components if those are far enough from each others, it means that global minimizers are obtained by splitting the mass into smaller sets.

The question of whether equality can be achieved in \eqref{eq:opt2} is, as far as the authors know, an open and interesting question.

\medskip

\noindent {\bf  Notation:} Throughout the paper, we denote $\omega_n=|B_1|$ the volume of the unit ball and recall that $n\omega_n=P(B_1)$ and we denote $\nu_n=n\omega_n^{1/n}$ the isoperimetric constant which is such that
\begin{equation}\label{eq:iso}
\nu_n |E|^{\frac{n-1}{n}} \leq  P(E)
\end{equation}
for all $E\subset \R^n$, with equality when $E=B_r$.
\medskip

\noindent {\bf Outline of the rest of the paper:}
The existence of bounded minimizers for small $\gamma:=\alpha m^{\frac{1-s}{n}}$ (Theorem \ref{thm:existence}) is proved in Section \ref{sec:existence}.
In Section \ref{sec:ndg}, we establish two key properties of these minimizers: a classical non-degeneracy estimate (for all $\gamma$) and the regularity of the minimizer (for small $\gamma$).
These results are used to prove Theorem \ref{thm:ball} (minimizers are balls for small $\gamma$) in Section \ref{sec:ball}.

Section \ref{sec:nonball} is devoted to further investigation of the role of the ball when $\gamma$ is large
(Theorems \ref{thm:nonball} and \ref{thm:beta1beta2}).
In Section \ref{sec:G}, we prove the existence of a minimizer for all $\gamma$ for the generalized minimization problem \eqref{eq:minG} (Theorem \ref{thm:G}).
Finally, Section \ref{sec:asJ} is devoted to Propositions \ref{prop:as} and \ref{prop:JGN} which relate our minimization problem to the interpolation inequality \eqref{eq:opt2}.

%\section{Preliminary results}
%\begin{lemma}\label{lem:PsP}
%There exists a constant $C$ such that 
%for all set $F\subset\R^n$ with finite perimeter we have
%\begin{equation}\label{eq:PsP}
% P_K(F) \leq P_s(F) \leq \mu_{n,s}|F|^{1-s} P(F)^s
% \end{equation}
%and 
%\begin{equation}\label{eq:PJP}
% P_K(F) \leq  C P(F)^s (|F|^{1/n} P(F)) ^{1-s} \leq C P(F) |F|^{\frac{1-s}{n}} 
%\end{equation}
%\end{lemma}
%\begin{proof}
%Inequality \eqref{eq:PsP} together with the isoperimetric inequality then implies
%$$
% P_J(F) \leq  C P(F)^s (|F|^{1/n} P(F)) ^{1-s}
%$$
%and \eqref{eq:PJP} follows.
%\end{proof}

\section{Proof of Theorem \ref{thm:existence}}\label{sec:existence}

%We note that \eqref{eq:GN} together with Young's inequality imply that
%\begin{equation}\label{eq:Emin}
%\J_{\alpha}(F) \geq  P(F)-\alpha C |F|^{1-s} P(F)^s \geq \frac 1 2 P(F) - C \alpha^{\frac{1}{1-s}} |F|
%\end{equation}
%and so $ \inf_{F\in \E_m} \J_\alpha (F) >- C\alpha^{\frac{1}{1-s}} m$. 

We already explained in the introduction that inequality \eqref{eq:Jlower} implies the existence of a minimizing sequence $F_k$ such that $|F_k|=m$ and $\lim_{k\to \infty} \J_\alpha(F_k) =  \inf\{  \J_\alpha (F)\,;\, |F|=m\}$. 
Inequality \eqref{eq:Jlower}  also implies that $P(F_k) $ is bounded with respect to $k$. 
The key step in the proof of Theorem \ref{thm:existence} is thus to show that we can always modify the sequence $F_k$ to construct a minimizing sequence which is bounded.
This will follow from the following proposition:
\begin{proposition}\label{prop:wF}
Assume that $F$ is such that 
\begin{equation}\label{eq:Br} 
\J_\alpha (F) \leq \J_\alpha(B_{r_0})
\end{equation}
where $|B_{r_0}|=|F|=m$. 
There exists  $\gamma_0$ depending only on $n$ such that if $\gamma :=   \alpha\, m^{\frac{1-s}{n}} \leq\gamma_0  $ then there exists $\widetilde F \subset B_{2r_0}$ with
$$|\widetilde F | = |F| =m$$ 
and
$$ \J_\alpha (\widetilde F) \leq \J_\alpha (F)$$
\end{proposition}

This proposition implies that if  $\alpha\, m^{\frac{1-s}{n}}$ is small enough (depending only on $n$) then
there exists a minimizing sequence $\widetilde F_k$ satisfying
$$ \widetilde F_k \subset B_1.$$
It is not difficult to prove that $ \widetilde F_k $ convergence strongly in $L^1(\R^n)$ to $E$ and 
Theorem \ref{thm:existence} follows from the lower semi-continuity of the perimeter and the following classical lemma (whose proof is presented in Appendix \ref{sec:CVPK}):
\begin{lemma}\label{lem:CVPK}
Assume that $\chi_{F_K}$ converges strongly in $L^1(\R^n)$ to $\chi_E$ and that $P(F_K)\leq C$. Then
$$ \lim_{k\to\infty} P_K(F_k) = P_K(E).$$
\end{lemma}

\medskip

Before proving Proposition \ref{prop:wF}, we prove the following preliminary lemma:
\begin{lemma}\label{lem:prelbound}
Let $F$ be such that 
$\J_\alpha (F) \leq \J_\alpha(B_{r_0})$ with $|F|=|B_{r_0}|=m$ and denote $\gamma =  \alpha  \, m^{\frac{1-s}{n}}$. Then there exists a constant $C$ depending only on $n$ such that: 
\begin{equation}\label{eq:perbound}
P(F) \leq C \left(1+ (n\omega_n\gamma)^{\frac{1}{1-s}} \right) m^\frac{n-1}{n} ,
\end{equation}
\begin{equation}\label{eq:Dbound}
D(F) := \frac{P(F) - P(B_{r_0}) }{P(B_{r_0})} \leq  C\gamma\left(1+ (n\omega_n \gamma)^{\frac{1}{1-s}} \right)^s ,
\end{equation}
%Then there exists $\gamma_0$ (depending on $n$ and $s$) such that
%if $ \gamma =  \alpha\, m^{\frac{1-s}{n}} \leq \gamma_0 \, $ then
and up to a translation 
\begin{equation}\label{eq:diffbound}
|F\setminus B_{r_0}|\leq |F\Delta B_{r_0}| \leq C m \sqrt { \gamma} \left(1+ (n\omega_n \gamma)^{\frac{1}{1-s}} \right)^{\frac s 2} .
\end{equation}
In particular, if $\gamma\leq (n\omega_n)^{-1}$, then
\begin{equation}\label{eq:bd1}
P(F) \leq C m^\frac{n-1}{n} , \qquad |F\setminus B_{r_0}| \leq C m \sqrt { \gamma} 
\end{equation}
for some constant $C$ depending only on $n$.
\end{lemma}

\begin{proof}[Proof of Lemma \ref{lem:prelbound}]
Using \eqref{eq:Br} and \eqref{eq:Jlower}, we get:
$$ 
 \frac{1}{2} P(F) -   \big( n\omega_n \alpha \big)^{\frac{1}{1-s}}  |F| \leq \J_\alpha(F) \leq \J_\alpha(B_{r_0}) \leq P(B_{r_0}) = \nu_n m^{\frac{n-1}{n}}.
$$
%Using   and the inequality \eqref{eq:GN}, we get:
%\begin{align*}
%P(F) & \leq P(B_{r_0}) + \alpha s(1-s) P_K(F)\\
%& \leq P(B_{r_0}) + \alpha C  P(F)^s |F|^{1-s} \\
%& \leq P(B_{r_0}) + \frac 1 2 P(F) + C \alpha ^{\frac{1}{1-s}} |F| 
%\end{align*}
and so
$$
P(F) \leq 2 \nu_n m^{\frac{n-1}{n}}  + 2 \big(n\omega_n \alpha \big)^{\frac{1}{1-s}}  m
$$
which implies \eqref{eq:perbound}.
 \medskip

Next, using  \eqref{eq:GN}, \eqref{eq:Br} and  \eqref{eq:perbound} we can write:
\begin{align*}
P(F) - P(B_{r_0}) 
& \leq  \alpha s(1-s)  P_K(F) \\
& \leq \alpha 2^{1-s} n\omega_n  |F|^{1-s} P(F)^s\\
%& \leq \alpha C |F|^{1-s} P(B_{r_0})^s\\
& \leq C \alpha   \left(1+ (n\omega_n \gamma)^{\frac{1}{1-s}} \right)^s m^{\frac{n-s}{n}}\\
& \leq C \gamma \left(1+ (n\omega_n  \gamma)^{\frac{1}{1-s}} \right)^s m^{\frac{n-1}{n}}
\end{align*}
and \eqref{eq:Dbound} follows.
%$$ 
%D(F) = \frac{P(F) - P(B_{r_0}) }{P(B_{r_0})} \leq C\gamma\left(1+ (2c_n \gamma)^{\frac{1}{1-s}} \right)^s .
%$$
We now conclude by using the following quantitative isoperimetric inequality \cite{Fusco2008}:
$$ 
\inf_{x\in \R^n}\frac{|F \Delta B_{r_0}(x)|}{|B_{r_0} (x)|} \leq C \sqrt{D(F)} .
$$
%Since the symmetric difference is defined by $F\Delta G = (F\setminus G) \cup (G\setminus F)$, \eqref{eq:diffbound} follows.
\end{proof}

\begin{proof}[Proof of Proposition \ref{prop:wF}]
Throughout the proof, we fix $r_0>0$ such that $|B_{r_0}| =m$ and we recall that $ \gamma :=\alpha\, m^{\frac{1-s}{n}}.$
Since we want to prove the result for small enough $\gamma$, we will always assume that $\gamma\leq (n\omega_n)^{-1}$.
Given $\rho \in (r_0,1)$, we define
$$ F_1(\rho) = F\cap B_\rho, \qquad F_2(\rho) = F\setminus B_\rho$$
(we drop the dependence of $F_1$ and $F_2$ on $\rho$ below)
and define by $A(\rho) $ the area of $F\cap \partial B_\rho$:
$$A(\rho) = \H^{n-1} (F\cap \p B_\rho).$$
We note that for a.e. $\rho\in(r_0,2r_0)$, we have
$$ 
\Sigma(\rho) : = P(F_1) + P(F_2) -P(F) = 2 A(\rho).
$$

We now adapt an argument from \cite{muratov-general}: 
The proof of Proposition \ref{prop:wF} is divided in two cases, depending on whether $\Sigma(\rho) \geq \frac 1 4  P(F_2)$ for all $\rho$ or  $\Sigma(\rho) < \frac 1 4  P(F_2)$ for some $\rho$.
\medskip

\noindent{\bf Case 1:} Assume that
\begin{equation}\label{eq:case1}
 \Sigma(\rho) \geq \frac 1 4  P(F_2(\rho))\qquad  \mbox{ for all } \rho\in (r_0,1).
\end{equation}
Introducing $U(\rho) = |F_2(\rho)|$, we note that
$$
\frac{ d}{d\rho} U(\rho) = A(\rho)=- \frac 1 2 \Sigma(\rho) 
$$
and so \eqref{eq:case1}
together with the isoperimetric inequality \eqref{eq:iso} imply
\begin{align*}
\frac{ d}{d\rho} U(\rho)
& \leq -\frac 1 8 P(F_2(\rho)) \\
& \leq -c U(\rho)^{\frac {n-1}{n}}  \qquad \mbox{ for all } \rho\in (r_0,1).
\end{align*}
In particular, this implies that there exists $c_1>0$ such that if $ U(r_0) \leq c_1 r_0^n$ then $U(2r_0)=0$.
The bound \eqref{eq:bd1} implies 
$$U(r_0) = |F\setminus B_{r_0}| \leq Cr_0^n \sqrt \gamma$$
when $\gamma\leq 1$, and so $U(r_0) \leq c_1 r_0^n$  provided $\gamma \leq \gamma_0$ small enough (depending only on $n$). 
In that case, we find $U(2r_0)=0$, and so $F\subset B_{2r_0}$ (up to a set of measure $0$) and the Proposition follows with $\widetilde F  = F$.

\medskip
\noindent{\bf Case 2:} Assume now that
\begin{equation}\label{eq:case2}
 \Sigma(\rho_0) \leq \frac 1 4  P(F_2(\rho_0)) \mbox{ for some } \rho_0\in (r_0,2r_0).
\end{equation}
We now denote $F_1=F_1(\rho_0)$ and $F_2=F_2(\rho_0)$. The definition of $\Sigma(\rho_0)$ and \eqref{eq:case2} imply
\begin{equation}\label{eq:bdP1}
P(F_1) \leq  P(F) - \frac{3}{4} P(F_2)
\end{equation}
and so
$$
\J_\alpha (F_1) \leq  \J_\alpha (F) - \frac{3}{4} P(F_2) +\alpha s (1-s) (P_K(F)-P_K(F_1)).
$$ 
Furthermore, the fact that $\chi_F=\chi_{F_1}+ \chi_{F_2}$ and the triangle inequality imply
$$ P_K(F) \leq P_K(F_1) + P_K(F_2)$$
and so (using  \eqref{eq:KP})
\begin{align*}
\J_\alpha (F_1) & \leq  \J_\alpha (F)  - \frac{3}{4} P(F_2) +\alpha s(1-s)P_K(F_2)\\
%& \leq  \J_\alpha (F)  - \frac{3}{4} P(F_2) +\alpha c_n P(F_2)^s |F_2|^{1-s}\\
& \leq  \J_\alpha (F)  - \frac{1}{4} P(F_2) +   \big( n\omega_n \alpha \big)^{\frac{1}{1-s}} |F_2| 
\end{align*}
Finally, denoting $m_2 = |F_2|$ and using the isoperimetric inequality \eqref{eq:iso}, we get
\begin{align*}
\J_\alpha (F_1) 
& \leq  \J_\alpha (F)  -\frac{ \nu_n}{4} m_2^{\frac{n-1}{n}} +\big( n\omega_n \alpha \big)^{\frac{1}{1-s}} m_2  \\
& \leq  \J_\alpha (F)  -\frac{ \nu_n}{4} m_2^{\frac{n-1}{n}} +\big( n\omega_n \gamma \big)^{\frac{1}{1-s}}  m_2^\frac{n-1}{n} 
%& \leq  \J_\alpha (F)  -\frac{ \nu_n}{4} m_2^{\frac{n-1}{n}} +C\gamma^{\frac 1 n} m_2^\frac{n-1}{n} 
\end{align*}
and so
\begin{equation}\label{eq:EaF1}
\J_\alpha (F_1) \leq  \J_\alpha (F)  - \frac{\nu_n}{8} m_2^{\frac{n-1}{n}} 
\end{equation}
provided $n\omega_n\gamma\leq \left(\frac{\nu_n}{8}\right)^{1-s}$, which is satisfied for all $s$ if $\gamma\leq \gamma_0$ small enough (depending only on $n$).

So, the set $F_1$ has a lower energy than $F_1$. However, we cannot take $\widetilde F=F_1$ since $|F_1| = m_1 = m-m_2 \neq m$. So we define
$$
\widetilde F = (1+t)^{\frac 1 n} F_1 , \mbox{ with $t$ such that $|\widetilde F |=m$.}
$$
Note that we have in particular $t = \frac{m}{m_1} - 1=\frac{m_2}{m_1}$.
Inequality \eqref{eq:bd1} implies $m_2 \leq m/2$ if $\tilde \gamma_0$ is small enough and so $m_1=m-m_2\geq m/2$. In turn, we get (using  \eqref{eq:bd1}  again)
$$ t = \frac{m_2}{m_1} \leq C \sqrt \gamma \frac{m}{m_1} \leq 2 C \sqrt{\gamma}.$$

Furthermore, \eqref{eq:EaF1} implies
\begin{align*}
\J_\alpha(\widetilde F)
& \leq  \J_\alpha (F)  -  \frac{\nu_n}{8} m_2^{\frac{n-1}{n}}  + P(\widetilde F)-P(F_1)+\alpha s (1-s)[P_K(F_1) - P_K(\widetilde F)]\\
& \leq  \J_\alpha (F)  -  \frac{\nu_n}{8} m_2^{\frac{n-1}{n}}  + \left[(1+t)^{\frac {n-1} n} -1\right] P(F_1)+\alpha s (1-s) [P_K(F_1) - P_K(\widetilde F)].
\end{align*}
Inequality \eqref{eq:bdP1} together with \eqref{eq:perbound} (with $\gamma\leq 1$) implies 
\begin{equation}\label{eq:P1bd}
P(F_1) \leq C m^{\frac{n-1}{n}}
\end{equation}
and we have the following lemma (proved below):
\begin{lemma}\label{PJ(tE)}
Recall that $z\mapsto K(|z|)$ is non-increasing. For all set $F_1$ with finite perimeter and for all $t\in [0,1]$, there holds:
$$ 
-\frac{C}{s(1-s)} t P(F_1)^s |F_1|^{1-s}\leq P_K(F_1) - P_K( (1+t)^{\frac 1 n} F_1) \leq \frac{C}{s(1-s)} t P(F_1)^s |F_1|^{1-s}.%leq C t^\frac 1n m_2  ^{\frac{n-1}{n}} + C m_2
$$
for some constant $C$ depending only on $n$.
\end{lemma}
Combining all of those inequalities yield
\begin{align*}
\J_\alpha(\widetilde F)
& \leq  \J_\alpha (F)  -  \frac{\nu_n}{8} m_2^{\frac{n-1}{n}}  + Ct  m^{\frac{n-1}{n}} + C \alpha t m^{s \frac{n-1}{n}} m^{1-s}\\
& \leq  \J_\alpha (F)  -  \frac{\nu_n}{8} m_2^{\frac{n-1}{n}}  + Ct  m^{\frac{n-1}{n}} + C \alpha t m^{\frac{n-1}{n}} m^{\frac{1-s}{n}}\\
& \leq  \J_\alpha (F)  -  \frac{\nu_n}{8} m_2^{\frac{n-1}{n}}  + Ct  m^{\frac{n-1}{n}} + C \gamma t m^{\frac{n-1}{n}}. \\
%^\frac 1n m_2  ^{\frac{n-1}{n}} + C \alpha m_2
\end{align*}
Finally, we write
$$t  m^{\frac{n-1}{n}} = t^{\frac1 n} (t m)^{\frac{n-1}{n}}\leq C \gamma^{\frac{1}{2n}} m_2^{\frac{n-1}{n}}$$
and we deduce that there exists $\gamma_0$ such that if $\gamma\leq \gamma_0$ then
$
\J_\alpha(\widetilde F)
 \leq  \J_\alpha (F)
$.
The result follows.
\end{proof}

\begin{proof}[Proof of Lemma \ref{PJ(tE)}]
Let $\rho=(1+t)^{\frac{1}{n}}\geq 1$. We write
$$ P_K(F) = \int_{\R^n} K(z) V_F(z)\, dx, \qquad V_{F}(z)=\int_{\R^n}{|\chi_{F}(x+z)-\chi_{F}(x)|dx}.
$$
Since 
\begin{equation}\label{eq:VF1}
V_{\rho F}(z) =\rho^n V_{F} (z/\rho),
\end{equation}
we can write
\[
P_K(F_1)-P_K(\rho F_1)=\int_{\R^n}{(K(z)-\rho^{2n}K(\rho z))V_{F_1}(z)dz}.
\]
Using the bounds 
\[
V_{F_1}(z)\leq 2|F_1|, \ \ V_{F_1}(z)\leq |z|\int_{\R^n}{D\chi_{F_1}}=|z|P(F_1),
\]
and $K(z)-K(\rho z)\geq 0$ (since $K(|x|)$ is non-increasing), we get
\begin{equation*}
\begin{aligned}
P_K(F_1)-P_K(\rho {F_1})&\leq \rho^{2n}\int_{\R^n}{(K(z)-K(\rho z))V_{F_1}(z)dz}\\
&\leq \rho^{2n}P(F_1)\int_{|z|\leq R}{(K(z)-K(\rho z))|z|dz}+ 2\rho^{2n}|F_1|\int_{|z|> R}{(K(z)-K(\rho z))dz}.
\end{aligned}
\end{equation*}
Since $K(x)\leq |x|^{-n-s}$, we get
\begin{equation*}
\begin{aligned}
\int_{|z|\leq R}{(K(z)-K(\rho z))|z|dz}&\leq (1-\rho^{-n-1})\int_{|z|\leq R}{|z|K(z)dz}+\rho^{-n-1}\int_{R\leq|z|\leq \rho R}{|z|K(z)dz}\\
&\leq (1-\rho^{-n-1})\frac{1}{1-s}R^{1-s}P(B_1)+\rho^{-n-1}\frac{1}{1-s}R^{1-s}(\rho^{1-s}-1)P(B_1)\\
&\leq P(B_1)\frac{1}{1-s}(1-2\rho^{-n-1}+\rho^{-n-s})R^{1-s},
\end{aligned}
\end{equation*}
and
\begin{equation*}
\begin{aligned}
\int_{|z|> R}{(K(z)-K(\rho z))dz}&\leq (1-\rho^{-n})\int_{|z|> R}{K(z)dz}+\rho^{-n}\int_{R\leq|z|\leq \rho R}{K(z)dz}\\
&\leq (1-\rho^{-n})\frac{1}{s}R^{-s}P(B_1)+\rho^{-n}\frac{1}{s}R^{-s}(1-\rho^{-s})P(B_1)\\
&\leq P(B_1)\frac{1}{s}(1-\rho^{-n-s})R^{-s}.
\end{aligned}
\end{equation*}
Using the fact that $\rho=(1+t)^{1/n}$ with $0\leq t\leq 1$, we deduce
\begin{equation*}
\begin{aligned}
P_K(F_1)-P_K(\rho {F_1})&\leq C(n)t \left(\frac{1}{1-s}P(F_1)R^{1-s}+2\frac{1}{s}|F_1|R^{-s}\right).
\end{aligned}
\end{equation*}
Optimizing with respect to $R$ by taking $R=\frac{2|F_1|}{P(F_1)}$ yields 
\[
P_K(F_1)-P_K((1+t)^{\frac{1}{n}} {F_1})\leq t\frac{C(n)}{s(1-s)}P(F_1)^s|F_1|^{1-s}
\]
which proves the right hand side inequality of Lemma \ref{PJ(tE)}.
\medskip

To prove the left hand side inequality, we write, for $\rho=(1+t)^{1/n}\in[1,2]$ (using \eqref{eq:VF1} and the fact that $x\mapsto K(|x|)$ is non-increasing): 
\begin{equation*}
\begin{aligned}
P_K(F_1)-P_K((1+t)^{\frac{1}{n}} {F_1})&=\int_{\R^n}{(K(z)-K(\rho z))V_{F_1}(z)dz}+(\rho^{-2n}-1)P_K(\rho F_1)\\
&\geq (\rho^{-2n}-1)P_K(\rho F_1)\\
&\geq -Ct\frac{t}{s(1-s)}P(\rho F_1)^s|\rho F_1|^{1-s}\\
&\geq -\frac{C}{s(1-s)}t P(F_1)^s|F_1|^{1-s}
\end{aligned}
\end{equation*}
and the proof is complete.
\end{proof}

\section{Non-degeneracy and regularity for the minimizer}\label{sec:ndg}
In this section, we prove some classical properties of the minimizers of $\J_\alpha$, namely non degeneracy and regularity. 
These properties play a key role in proof of Theorem \ref{thm:ball}.
We start with the following non-degeneracy result:
\begin{lemma}\label{nondegeneracyE}
Let $E$ be a minimizer of \eqref{eq:min}. 
There exists a constant $c$ depending only on $n$ such that
\[
|E\cap B_r(x_0)|\geq c\min \left\{ m, \frac{m}{(n\omega_n\gamma)^{\frac{n}{1-s}}}, r^n\right\}
, \qquad \mbox{ where } \gamma = \alpha\, m^{\frac{1-s}{n}}
\]
for every $x_0\in \partial E$ and $r>0$.
In particular, we have: 
\begin{equation}\label{eq:ndg}
|E\cap B_r(x_0)|\geq c r^n \qquad \mbox{ for all $r\leq C \min\{m^{1/n},(n\omega_n \alpha)^{-\frac{1}{1-s}}\} $}.
\end{equation}
\end{lemma}
Note that the result also holds, for small enough $r$, if $E$ is a local minimizer.

\begin{proof}
Given $r>0$
we define the  functions:
$$
A(r)=H^{n-1}(\partial B_r(x_0)\cap E),
\quad S(r)=P(E,B_r(x_0)),
\quad V(r)=|E\cap B_r(x_0)|,
$$
%We know $V(0)=0$ and $V(r)\leq |B_1|r^n$ is non-decreasing in $r$. We want to prove when $r<C_2r_0$,
%\[
%V(r)\geq Cr^n.\\
%\]
and the sets
\[
E'=E\setminus B_r(x_0), \qquad F=tE'
\]
where $t$ is chosen such that $|F|=|E|=m$, that is:
\[
t:=\left(\frac{m}{m-V(r)}\right)^{1/n}=(1+\lambda(r))^{1/n}, \qquad \lambda(r):=\frac{V(r)}{m-V(r)}.
\]
%with $0\leq \mu:=\frac{V(r)}{m-V(r)}< \frac{1}{2}$ when $r<\frac{1}{2}r_0$ and $n\geq 2$. Here $|B_{r_0}|=m$ and $r_0=(\frac{m}{|B_1|})^{1/n}$. 
Given $\bar r>0$, we assume that 
$$
V(\bar r)\leq \eps_0 \min\left\{ 1,\frac{1}{(n\omega_n \gamma)^{\frac{n}{1-s}}}\right\} m 
$$
for some $\eps_0<1/2$ to be chosen later. 
Since $r\mapsto V(r)$ is non-decreasing, we then have 
\begin{equation}\label{eq:eps0} 
V( r)\leq \eps_0 \min\left\{ 1,\frac{1}{(n\omega_n \gamma)^{\frac{n}{1-s}}}\right\} m  \quad\mbox{ for all $r\leq \bar r$,}
\end{equation}
and so in particular 
\begin{equation}\label{eq:muV}
 \lambda(r)=\frac{V(r)}{m-V(r)} \leq 2 \frac{V(r)}{m}\quad\mbox{ for all $r\leq \bar r$.}
 \end{equation}
 \medskip
 
Since $E$ is a volume constraint minimizer of $\J_\alpha$, we have
$\J_{\alpha}(E)\leq \J_{\alpha}(F)$,
which implies
\begin{align*}
P(E)-\alpha s(1-s)P_K(E)
&\leq P(tE')-\alpha s(1-s)P_K(tE')\\
&\leq t^{n-1}P(E')-\alpha s(1-s)P_K(tE')\\
&\leq (1+\lambda)^{\frac{n-1}{n}}P(E')-\alpha s(1-s) P_K(E')+\alpha s(1-s)\big[P_K(E')- P_K(tE')\big].
\end{align*}
Since $(1+\lambda)^{\frac{n-1}{n}}\leq 1+\frac{n-1}{n}\lambda \leq 1+\lambda $, we deduce
\begin{align}
S(r)-A(r) & = P(E)-P(E') \nonumber \\
&\leq \lambda P(E')+\alpha s(1-s)[ P_K(E)- P_K(E')]+\alpha s(1-s)\big[P_K(E')- P_K(tE')\big].\label{eq:mult}
\end{align}
We thus need to bound the last two term in the right hand side.
First, using Lemma \ref{PJ(tE)}, Young's inequality and  \eqref{eq:muV} we can write
\begin{align}
\alpha s(1-s) [P_K(E')- P_K(tE')] & = \alpha s(1-s) \left[P_K(E')- P_K\left((1+\lambda)^{\frac 1 n}E'\right) \right] \nonumber \\
& \leq C\alpha \lambda P(E')^s|E'|^{1-s} \nonumber \\
& \leq C \alpha \lambda P(E')^sm^{1-s}\nonumber  \\
&\leq C \lambda P(E')+C \lambda(n\omega_n \alpha ) ^{\frac{1}{1-s}}m\label{eq:PKEL}   \\
&\leq C \lambda P(E')+C  (n\omega_n  \alpha ) ^{\frac{1}{1-s}}V(r). \nonumber
\end{align}
Next, we write (using the definition of $P_K$ and \eqref{eq:KP}):
\begin{align*}
\alpha s(1-s)(P_K(E)-P_K(E'))&=2 \alpha s(1-s)\left(\int_{E}{\int_{E^c}{K(x-y)dxdy}}-\int_{E\setminus B_r}{\int_{(E\setminus B_r)^c}{K(x-y)dxdy}}\right)\\
&\leq 2 \alpha s(1-s)\int_{B_r\cap E}{\int_{(B_r\cap E)^c}{K(x-y)dxdy}}\\
&\leq \alpha s(1-s) P_K(B_r\cap E)\\
& \leq  \frac{1}{2} P(B_r\cap E) +   \big( n\omega_n \alpha \big)^{\frac{1}{1-s}} |B_r\cap E| \\
& \leq  \frac{1}{2}\big(A(r)+S(r)\big)  + \big( n\omega_n \alpha \big)^{\frac{1}{1-s}}V(r) .
\end{align*} 
Going back to \eqref{eq:mult}, we deduce:
$$
S(r)-A(r)
\leq C\lambda P(E')+  \frac{1}{2}\big(A(r)+S(r)\big) + C \big( n\omega_n  \alpha \big)^{\frac{1}{1-s}}V(r) 
$$
and so
%Plugging this inequality into \eqref{eq:SAS} implies:
%\begin{equation*}
%S(r)-A(r)
%\leq  2\lambda P(E')+\frac{1}{2}(S(r)+A(r))+ \big( C \alpha \big)^{\frac{1}{1-s}}V(r)
%\end{equation*} 
%that is 
$$
 S(r)
\leq  C\lambda P(E')+C A(r)+  C \big(n\omega_n  \alpha \big)^{\frac{1}{1-s}}V(r) 
$$
Finally, we have (using \eqref{eq:perbound}):
\[
P(E')= P(E)-S(r)+A(r)\leq C \left(1+ (n\omega_n  \gamma)^{\frac{1}{1-s}} \right) m^\frac{n-1}{n} 
+A(r),
 \]
and so
\begin{equation*}
\begin{aligned}
 S(r)
%&\leq 3\mu (Cm^{\frac{n-1}{n}}-S(r)+A(r))+C\mu \alpha^{\frac{1}{1-s}}m\\
%&+\frac{1}{2}(S(r)+A(r))+C_4\alpha^{\frac{1}{1-s}}V(r)\\
&\leq C A(r) +C \big( n\omega_n  \alpha \big)^{\frac{1}{1-s}}V(r)+C\lambda \left(1+ (n\omega_n  \gamma)^{\frac{1}{1-s}} \right) m^\frac{n-1}{n} \\
&\leq CA(r) +  C\left(1+ (n\omega_n  \gamma)^{\frac{1}{1-s}} \right) m^\frac{-1}{n} V(r)
\end{aligned}
\end{equation*} 
(where we used \eqref{eq:muV} and the fact that $\gamma = \alpha m^{\frac{1-s}{n}}$).
Using \eqref{eq:eps0}, this implies
\begin{equation}\label{eq:SA}
S(r) 
 \leq CA(r)+C\eps_0^{1/n} V(r)^\frac{n-1}{n} .
\end{equation}
 
In order to conclude, we combine \eqref{eq:SA} with the isoperimetric inequality \eqref{eq:iso}, which gives
\[
S(r)+A(r)\geq \nu_n V(r)^{\frac{n-1}{n}},
\]
and the fact that
\[
V'(r)=A(r) \qquad\mbox{for a.e. $r>0$},
\]
to conclude:
\begin{align*}
\nu_nV(r)^{\frac{n-1}{n}}&\leq CA(r)+C\eps_0^{1/n} V(r)^\frac{n-1}{n}  \\
&\leq CV'(r)+C\eps_0^{1/n} V(r)^\frac{n-1}{n}.
\end{align*}
%Since $\mu=\frac{V(r)}{m-V(r)}$, then when $r<C_2r_0$ with $C_2=\min\{\frac{C_n}{48C},\frac{1}{2}\}$, we have $V(r)<C_2^nm$ and thus
%\[
%6C\mu m^{\frac{n-1}{n}}\leq\frac{1}{4}C_nV(r)^{\frac{n-1}{n}}.
%\] 
%Also, when $\alpha m^{\frac{1-s}{n}}<\gamma_2$,
%\[
%2C \mu \alpha^{\frac{1}{1-s}}m\leq \frac{1}{12C}\alpha^{\frac{1}{1-s}}m^{\frac{1}{n}}C_nV(r)^{\frac{n-1}{n}}\leq \frac{1}{4}C_nV(r)^{\frac{n-1}{n}}.
%\]
%When $r<C_2r_0$ and $\alpha m^{\frac{1-s}{n}}<\gamma_2$,
%\[
%2C_4\alpha^{\frac{1}{1-s}}V(r)\leq \frac{1}{4}C_nV(r)^{\frac{n-1}{n}}.
%\]
%Thus from \eqref{V'(r)geq} we obtain the following ODE for $V(r)$ when $\gamma<\gamma_2$ and $0<r<C_2r_0$,
Choosing $\eps_0$ small enough, we thus have
$$
V'(r)\geq \frac{\nu_n}{2} V(r)^{\frac{n-1}{n}},\qquad\mbox{for a.e. $0<r<\bar r$}
$$
with $V(0)=0$, which implies
\[
V(\bar r)\geq \left(\frac{\nu_n}{2n}\right)^n \,{\bar r}^n 
\]
and the result follows.
\end{proof}

Proceeding similarly, we can also prove the following non-degeneracy lemma for $E^c$.
\begin{lemma}\label{nondegeneracyEC}
Let $E$ be a minimizer of \eqref{eq:min}. There exists a constant $c$ depending only on $n$ such that
\[
|E^c \cap B_r(x_0)|\geq c \min \left\{m, \frac{m}{(n\omega_n \gamma)^\frac{n}{1-s}}  ,  r^n  \right\} 
\]
for every $x_0\in \partial E$ and $r>0$. 
\end{lemma}
Note that the result also holds, for small enough $r$, if $E$ is a local minimizer.

\begin{proof}
Given $r>0$, we define
$$
A(r)=H^{n-1}(\partial B_r(x_0)\cap E^c),
\quad
S(r)=P(E,B_r(x_0)),
\quad
V(r)=|E^c\cap B_r(x_0)|,
$$
%We know $V(0)=0$ and $V(r)\leq |B_1|r^n$ is monotone non-decreasing in $r$. We want to prove when $r<\frac{1}{2}r_0$,
%\[
%V(r)\geq Cr^n.\\
%\]
and define the set
\[
E'=E\cup B_r(x_0),
\quad F=tE'
\]
with $t$ chosen so that $|F|=|E|=m$, that is
$t=(\frac{m}{m+V(r)})^{1/n}\leq 1$. Let $t=(1-\lambda(r))^{1/n}$ with $\lambda(r):=\frac{V(r)}{m+V(r)}$.
Since $E$ is a minimizer of $\J_\alpha$, we have
\[
\J_{\alpha}(E)\leq \J_{\alpha}(F),
\]
which implies

\begin{equation*}
\begin{aligned}
P(E)-\alpha s(1-s ) P_K(E)&\leq t^{\frac{n-1}{n}}P(E')-\alpha s(1-s) P_K(tE')\\
&\leq P(E')-\alpha s(1-s) P_K(E')+\alpha s(1-s) [P_K(E')-  P_K(tE')].\\
\end{aligned}
\end{equation*} 
With Lemma \ref{PJ(tE)} (for $t\leq 1$) and $\lambda(r)\leq \frac{V(r)}{m}$ we have
\begin{equation*}
\begin{aligned}
\alpha s(1-s)(P_K(E')- P_K(tE'))&\leq C\alpha \lambda P(E')^s|E'|^{1-s}\\
&\leq C\lambda P(E')+C\lambda (n\omega_n\alpha)^{\frac{1}{1-s}}m\\
&\leq C\lambda P(E')+C (n\omega_n\alpha)^{\frac{1}{1-s}}V(r).
\end{aligned}
\end{equation*} 

Therefore, using the definitions of $S(r), A(r)$ and $V(r)$,
\begin{equation*}
\begin{aligned}
S(r)-A(r)&=P(E)-P(E')\\
&\leq \alpha s(1-s)(P_K(E)- P_K(E'))+\alpha s(1-s)(P_K(E')- P_K(tE'))\\
&\leq \alpha s(1-s)P_K(E^c\cap B_r(0))+C\lambda P(E')+C (n\omega_n\alpha)^{\frac{1}{1-s}}V(r)\\
%&\leq \alpha s(1-s) P_s(E^C\cap B_r(0))\\
%&\leq \alpha C(A(r)+S(r))^sV(r)^{1-s}\\
&\leq \frac{1}{2}(A(r)+S(r))+C\lambda P(E')+C(n\omega_n \alpha)^{\frac{1}{1-s}}V(r).
\end{aligned}
\end{equation*} 
Thus,
\begin{equation}\label{P(E,B_r)}
\begin{aligned}
S(r)+A(r)&\leq 4A(r)+C\lambda P(E')+C(n\omega_n\alpha)^{\frac{1}{1-s}}V(r).
\end{aligned}
\end{equation} 
Finally, we have (using \eqref{eq:perbound}):
\[
P(E')= P(E)-S(r)+A(r)\leq C \left(1+ (n\omega_n \gamma)^{\frac{1}{1-s}} \right) m^\frac{n-1}{n} 
+A(r),
 \]
and so
\begin{equation*}
\begin{aligned}
S(r)+A(r)&\leq CA(r)+C\left(1+ (n\omega_n \gamma)^{\frac{1}{1-s}} \right) m^\frac{-1}{n} V(r),
\end{aligned}
\end{equation*} 
(where we used $\lambda(r)\leq \frac{V(r)}{m}$ and the fact that $\gamma = \alpha m^{\frac{1-s}{n}}$).

By the isoperimetric inequality and using the fact that $V'(r)=A(r)$, we obtain
$$
\nu_nV(r)^{\frac{n-1}{n}}\leq S(r)+A(r)\leq CV'(r)+C\left(1+ (n\omega_n \gamma)^{\frac{1}{1-s}} \right) m^\frac{-1}{n} V(r).
$$
If $V(r)\leq \eps_0\min\left\{ m, \frac{m}{(n\omega_n\gamma)^{\frac{n}{1-s}}} \right\}$, we deduce (since $r\to V(r)$ is non-decreasing)
$$
\nu_nV(s)^{\frac{n-1}{n}}\leq CV'(s)+C\eps_0^{1/n}V(s)^{\frac{n-1}{n}} \quad\mbox{ for all $s\leq r$}
$$
and so choosing $\eps_0$ small enough
$$
 V'(s) \geq \frac{\nu_n}{2}V(s)^{\frac{n-1}{n}}\quad\mbox{ for all $s\leq r$}
$$
which implies that $V(r) \geq cr^n$. We thus have, for all $r$, 
$$ 
V(r) \geq \min \left\{ c  r^n ,  \eps_0 m, \eps_0\frac{m}{(n\omega_n\gamma)^{\frac{n}{1-s}}}  \right\}
$$
and the result follows.
\end{proof}

The non-degeneracy estimate allows us to prove:
\begin{proposition}\label{lem:indecomposable}
Let $E$ be a minimizer of \eqref{eq:min}. 
Then $E$ is bounded, indecomposable and its diameter satisfies
\begin{equation}\label{eq:diam}
\mathrm{diam}(E) \leq  C\max\left\{1,(n\omega_n\gamma)^{\frac{1}{1-s}}\right\}^{n-1}  m^{1/n}
\end{equation}
\end{proposition}
\begin{proof}
Let $\bar r_0 = \min\{1,(n\omega_n \gamma)^{-\frac{1}{1-s}}\} m^{1/n}$, then the non-degeneracy estimate \eqref{eq:ndg}
implies
$$
|E\cap B_{\bar r_0}(x_0)|\geq c \bar r_0^n  \qquad \mbox{ for all } x_0\in E.
$$
In particular we can cover $E$ with at most $N:=C \frac{m}{\bar r_0^n} = C \max\{1,(n\omega_n \gamma)^{\frac{1}{1-s}}\}^n  $ ball $B_{\bar r_0}$ and $E$ is thus bounded.

To prove that the minimizer $E$ is indecomposable, suppose the opposite is true and that there exist two sets of finite perimeter $E_1$ and $E_2$ such that $E_1\cap E_2=\emptyset$ and $E=E_1\cup E_2$ with $P(E) = P(E_1) + P(E_2)$. Since $E$ is bounded when $\gamma<\gamma_0$, we have $E_1$ and $E_2$ are bounded. Define $E_R=E_1\cap (E_2+e_1 R)$. We have $|E_R|=m$ and $P(E_R) = P(E)$ for $R>0$ sufficiently large, and the nonlocal energy decreases
\[
P_K(E_R)=P_K(E_1)+P_K(E_2)-2\int_{E_1}{\int_{E_2+e_1R}K(x-y)dxdy},
\]
which contradicts the minimizing property of $E$.

Finally, proceeding as  in \cite{muratov-general} Lemma 7.2 we can show that the diameter of $E$ is bounded by $CN \bar r_0$ which leads to \eqref{eq:diam}.

\end{proof}

Classically, the non-degeneracy results also imply the following bound, which will be useful in the proof of Theorem \ref{regularity} below:
\begin{corollary}\label{P(E,B_r)bdd}
If $E$ is a minimizer of the energy $\J_{\alpha}(F)= P(F)-\alpha s(1-s)P_K(F) $, with the constraint $|F|=m$, then
, then there exists a universal constant $C_3>0$ so that for every $x\in \partial E$,
\[
P(E,B_r(x))\leq  C\left(1+(n\omega_n \gamma)^\frac{1}{1-s} \right) r^{n-1} , \mbox{ for all $r>0$.}
\]
\end{corollary}
\begin{proof}
Using \eqref{P(E,B_r)}, we obtain (using the fact that $V(r)\leq m$ and $V(r)\leq |B_r|$):
\begin{align*}
P(E,B_r(x))=S(r) 
& \leq 3A(r)+(n\omega_n\alpha)^{\frac{1}{1-s}}V(r)\\
& \leq 3P(B_r)+(n\omega_n \alpha)^{\frac{1}{1-s}} m^{\frac 1 n} |B_r|^{\frac{n-1}{n}}\\
& \leq C r^{n-1} + C(n\omega_n \gamma)^\frac{1}{1-s}   r^{n-1}.
\end{align*}
and the result follows.
\end{proof}

Finally, we end this section with the following regularity result:
\begin{theorem}\label{regularity}
Let  $E$ be a minimizer of \eqref{eq:min} with $|E|=m$ and let $r_0$ be such that $|B_{r_0}| =m$.
%There exists $\gamma_1'>0$ depending only on $n$ and $s$ such that if $ \alpha\, m^{\frac{1-s}{n}} \leq \gamma_1'$,
Then the reduced boundary $\partial^*E$ is a $C^{1,\frac{1}{2}(1-s)}$ hypersurface in $\mathbb{R}^n$, and
\[
\mathcal H^k[\partial E\setminus \partial^*E]=0
\]
for all $k>n-8$ (where $\mathcal H^k$ denotes the $k-$Hausdorff measure). 

\item Furthermore, when $n\leq 7$ and $ \alpha\, m^{\frac{1-s}{n}} \leq \gamma_1'$ (small enough, but depending only on $n$ and $s$),
we can write
\[
\partial E=\{x_0+r_0(1+u(x))x:x\in \partial B_1\} 
\]
for some function $u\in C^{1,\frac{1}{2}(1-s)}(\partial B_1)$ with the $C^{1,\frac{1}{2}(1-s)}$ regularity constant depending on $n$ and $s$ only. 
\end{theorem}
\begin{proof}
Let $E^*=\frac{1}{r_0}E$ so that $|E^*|=|B_1|$. For every $x^*\in \partial E^*$, define
\[
\Psi(E^*,B_{\rho}(x^*))=P(E^*,B_{\rho}(x^*))-\inf\{P(F^*,B_{\rho}(x^*)),E^*\Delta F^*\subset B_{\rho}(x^*)\}.
\]
The first part of Theorem \ref{regularity} follows from Theorem 1 in \cite{Tamanini1982} if we prove that
\begin{equation}\label{Psi(E^*,B_rho)}
 \Psi(E^*,B_{\rho}(x^*))\leq C\rho^{n-s}=C\rho^{n-1+2(\frac{1}{2}(1-s))}.
 \end{equation} 
%for every $x^*\in \partial E^*$ and $\rho\in(0,\frac{1}{2})$, with the constant $C$ being a universal constant. 
%implies the reduced boundary $\partial^*E^*$ is $C^{1,\frac{1}{2}(1-s)}$ hypersurface in $\mathbb{R}^n$, and
%\[
%H_k[E^*\setminus \partial^*E^*]=0
%\]
%for all $k>n-8$. Therefore, our aim is 
In order to prove \eqref{Psi(E^*,B_rho)}, we first note that 
%for every $x\in \partial E$, define
%\[
%\Psi(E,B_{\rho}(x))=P(E,B_{\rho}(x))-\inf\{P(F,B_{\rho}(x)),E\Delta F\subset B_{\rho}(x)\},
%\]
%and 
we have the following scaling property
\begin{equation}\label{eq:scaling2}
\Psi(E^*,B_{\frac{\rho}{r_0}}(x^*))=\frac{1}{r_0^{n-1}}\Psi(E,B_{\rho}(x))
\end{equation}
with $x=r_0x^*$. 
We then fix a set $F\subset \mathbb{R}^n$ such that $E\Delta F\subset B_{\rho}(x)$. Without loss of generality, we may assume   
\begin{equation}\label{WLOGF<E}
P(F,B_{\rho}(x))\leq P(E,B_{\rho}(x))
\end{equation}
since otherwise \eqref{Psi(E^*,B_rho)} is trivial. We now set $F'=tF$ with $t$ chosen so that
$|F'|=|tF|=|E|=m$. In particular, we have
\begin{equation}\label{eq:tn}
|t^n -1 | \leq 2\left(\frac \rho {r_0}\right)^{n} \qquad \mbox{ if } \rho< \frac{1}{2}r_0
\end{equation}
and since $E$ is a minimizer of $\J_{\alpha}$ in $\E_m$,
we can write:
\[
P(E)-P(tF)\leq \alpha s(1-s)(P_K(E)-P_K(tF)).
\]
%The proof of \eqref{Psi(E^*,B_rho)} now depends on whether $t\leq 1$ or $t>1$.
%\medskip
%\noindent {\it Case 1:} If $t\leq 1$, then for every $\rho\in(0,\frac{1}{2}r_0)$, we have:
and so
\begin{equation*}
\begin{aligned}
P(E,B_{\rho}(x))-P(F,B_{\rho}(x))&=P(E)-P(tF)+P(tF)-P(F)\\
&= P(E)-P(tF)+(t^{n-1}-1)P(F)\\
&\leq \alpha s(1-s)(P_K(E)-P_K(tF))+(t^{n-1}-1)P(F)\\
&\leq \alpha s(1-s)(P_K(E)-P_K(F))+\alpha s(1-s)(P_K(F)-P_K(tF))+(t^{n-1}-1)P(F)
\end{aligned}
\end{equation*}
For the first term, we use \eqref{WLOGF<E} and Corollary \ref{P(E,B_r)bdd} to get
\begin{equation*}
\begin{aligned}
\alpha s(1-s) (P_K(E)-P_K(F))
&\leq \alpha C P(E\Delta F)^s|E\Delta F|^{1-s}\\
&\leq \alpha C (P(E,B_{\rho}(x))+P(F,B_{\rho}(x))+C\rho^{n-1})^{s}\rho^{n(1-s)}\\
&\leq C\alpha \left(1+ \gamma^\frac{1}{1-s} \right)^s \rho^{n-s}
\end{aligned}
\end{equation*}
for some constant $C$ depending on $n$ and $s$.
The second term is bounded by Lemma \ref{PJ(tE)} and \eqref{eq:tn}:
\begin{equation*}
\begin{aligned}
\alpha s(1-s) (P_K(F)-P_K(tF))
&\leq C \alpha |1-t^n|P(F)^s|F|^{1-s}\\
&\leq C \alpha \frac{\rho^n}{ m}P(E)^s\left(\frac{|E|}{t^n}\right)^{1-s}\\
&\leq C\alpha \left(1+  \gamma^\frac{1}{1-s} \right)^s \frac{\rho^n}{m}m^{\frac{n-1}{n}s}m^{1-s}t^{-n(1-s)}\\
&\leq C\alpha  \left(1+ \gamma^\frac{1}{1-s}  \right)^s\rho^{n-s}\rho^sr_0^{-s}\\
&\leq C\alpha  \left(1+  \gamma^\frac{1}{1-s} \right)^s\rho^{n-s}.
\end{aligned}
\end{equation*}
The last term satisfies
$$ (t^{n-1}-1)P(F) \leq C P(E) \left(\frac \rho {r_0}\right)^{n} \leq C\left(1+  \gamma^\frac{1}{1-s} \right) r_0^{n-1} \left(\frac \rho {r_0}\right)^{n} \leq C\rho^nr_0^{-1} .$$

We deduce:
\[
P(E,B_{\rho}(x))-P(F,B_{\rho}(x))\leq C\alpha \left(1+  \gamma^\frac{1}{1-s} \right)^s \rho^{n-s}+C \left(1+ \gamma^\frac{1}{1-s} \right)\rho^nr_0^{-1},
\]
and since this holds for all $F\subset \mathbb{R}^n$ such that $E\Delta F\subset B_{\rho}(x)$, we get
$$ 
\Psi(E,B_{\rho}(x)) \leq C\alpha  \left(1+  \gamma^\frac{1}{1-s} \right)^s \rho^{n-s} +C \left(1+  \gamma^\frac{1}{1-s} \right) \rho^nr_0^{-1} \qquad \forall \rho < \frac{r_0}{2}.
$$
The scaling property \eqref{eq:scaling2} implies 
\begin{equation}\label{Psi_t>1}
\begin{aligned}
\Psi(E^*,B_{\frac{\rho}{r_0}}(x^*))&=\frac{1}{r_0^{n-1}}\Psi(E,B_{\rho}(x))\\
%&\leq Cr_0^{-n+1}(\alpha\left(1+  \gamma^\frac{1}{1-s} \right)^s \rho^{n-s}+\left(1+  \gamma^\frac{1}{1-s} \right)\rho^nr_0^{-1})\\
&\leq C\alpha r_0^{1-s} \left(1+  \gamma^\frac{1}{1-s} \right)^s \left(\frac{\rho}{r_0}\right)^{n-s}+C\left(1+  \gamma^\frac{1}{1-s} \right) \left(\frac{\rho}{r_0}\right)^{n}\\
&\leq C\gamma \left(1+  \gamma^\frac{1}{1-s} \right)^s \left(\frac{\rho}{r_0}\right)^{n-s}+C\left(1+  \gamma^\frac{1}{1-s} \right) \left(\frac{\rho}{r_0}\right)^{n-s}\\
%&\leq C(1+\alpha r_0^{1-s})\left(\frac{\rho}{r_0}\right)^{n-s}\\
&\leq C\left(1+  \gamma^\frac{1}{1-s} \right)\left(\frac{\rho}{r_0}\right)^{n-s} \qquad \forall \rho < \frac{r_0}{2}
\end{aligned}
\end{equation}
which completes the proof of \eqref{Psi(E^*,B_rho)} and the first part of the theorem  follows from Theorem 1 in \cite{Tamanini1982}.

\medskip

% and it implies the reduced boundary $\partial^* E^*$ is $C^{1,\frac{1}{2}(1-s)}$ with the regularity constant depending on $n$ and $s$ only, and
%\[
%H_k[E^*\setminus \partial^*E^*]=0
%\]
%for all $k>n-8$. 

Next, we note that since $E$ is the minimizer, we have (using \eqref{eq:Dbound} with $\gamma\leq (n\omega_n)^{-1}$):
\[
D(E^*)=D(E)=\frac{P(E)-P(B_{r_0})}{P(B_{r_0})}\leq C \gamma
\]
which implies
\[
|E^*\Delta B_1|\leq C\sqrt{ \gamma}
\]
Therefore, when $n\leq 7$, we can apply Lemma 6.4 in \cite{muratov-general} to prove that there exists $\gamma_1'>0$ so that when $\gamma<\gamma_1'$,
\[
\partial {}E^*=\{x_0+(1+u(x))x:x\in \partial B_1\} 
\]
for $x_0\in \mathbb{R}^n$ the barycenter of $E^*$ and some function $u\in C^{1,\frac{1}{2}(1-s)}(\partial B_1)$ with the $C^{1,\frac{1}{2}(1-s)}$ regularity constant depending on $n$ and $s$ only. This implies that
\[
\partial E=\{x_0+r_0(1+u(x))x:x\in \partial B_1\} 
\]
and concludes the proof of Theorem \ref{regularity}. 
\end{proof}

\section{Proof of Theorem \ref{thm:ball}}\label{sec:ball}
Throughout this section, we assume that $E$ is a minimizer of $\J_\alpha$ with $|E|=m$ and that $\gamma =\alpha\, m^{\frac{1-s}{n}} \leq \gamma_1'$ so that we can write
\[
\partial E=\{r_0(1+u(x))x:x\in \partial B_1\} 
\]
for some function $u\in C^{1,\frac{1}{2}(1-s)}(\partial B_1)$ with the $C^{1,\frac{1}{2}(1-s)}$ regularity constant depending on $n$ and $s$ only (see Theorem \ref{regularity}).

We note that when $\gamma  \leq (n\omega_n)^{-1}$, \eqref{eq:Dbound} and \eqref{eq:diffbound} implies
\[
D(E):=\frac{P(E)-P(B_{r_0})}{P(B_{r_0})}\leq C\gamma \leq C \alpha m^{\frac{1-s}{n}},
\]
and 
\begin{equation}\label{eq:L1convergence}
|E\Delta B_{r_0}|\leq Cm\sqrt{\gamma}\leq Cm\sqrt{\alpha m^{\frac{1-s}{n}}}.
\end{equation}
Using this, we can prove the following estimates on $u$ in $W^{1,\infty}(\partial  B_1)$:
\begin{lemma}\label{Linfinitybound}
Let $E$ be a minimizer of \eqref{eq:min} such that
\[
\partial E=\{r_0(1+u(x))x:x\in \partial B_1\} 
\]
for some function $u\in C^{1,\frac{1}{2}(1-s)}(\partial B_1)$. Then there exists a constant $\gamma_1''>0$ depending only on $n$ and $s$, such that when $\gamma=\alpha m^{\frac{1-s}{n}}\leq \gamma_1''$,
\begin{equation}\label{eq:Linfty}
\| u\|_{L^{\infty}(\partial B_1)}\leq \frac{3}{20n},
\end{equation}
and
\begin{equation}\label{eq:gLinfty}
\| \nabla u\|_{L^{\infty}(\partial B_1)}\leq \frac{1}{2}.
\end{equation}
\end{lemma}
\begin{proof}
Given $x_0\in \partial B_1$, we set $y_0=(1+u(x_0))x_0\in \partial E$. If $\rho=r_0u(x_0)>0$, then consider the set $E\cap B_{\rho/2}(y_0)\subset E\Delta B_{r_0}$. Lemma \ref{nondegeneracyE} implies,
\[
|E\cap B_{\rho/2}(y_0)|\geq C(\rho/2)^n=Cr_0^nu(x_0)^n,
\] 
and so \eqref{eq:L1convergence} yields:
\[
Cr_0^nu(x_0)^n\leq Cm\sqrt{\alpha m^{\frac{1-s}{n}}},
\]
which implies
\[
u(x_0)\leq C(\alpha m^{\frac{1-s}{n}})^{\frac{1}{2n}}\leq \frac{3}{20n},
\]
when $\gamma=\alpha m^{\frac{1-s}{n}} \leq \gamma_1''$ small enough (only depending on $n$ and $s$).
If $\rho=r_0u(x_0)\leq 0$, then consider the set $E^C \cap B_{\rho/2}(y_0)\subset E\Delta B_{r_0}$. By Lemma \ref{nondegeneracyEC} we have:
\[
|E^c \cap B_{\rho/2}(y_0)|\geq C(\rho/2)^n=Cr_0^n|u(x_0)|^n,
\] 
and  \eqref{eq:L1convergence}  gives:
\[
Cr_0^n|u(x_0)|^n\leq Cm\sqrt{\alpha m^{\frac{1-s}{n}}},
\]
which implies
\[
|u(x_0)|\leq C(\alpha m^{\frac{1-s}{n}})^{\frac{1}{2n}}\leq \frac{3}{20n},
\] 
when $\gamma\leq \gamma_1''$ small enough.
This proves \eqref{eq:Linfty}.

\medskip

To prove \eqref{eq:gLinfty}, we proceed by contradiction by  assuming that there exists $x_0\in \partial B_1$ such that $|\nabla u(x_0)|>\frac{1}{2}$. By Theorem \ref{regularity} we have $u\in C^{1,\frac{1-s}{2}}(\partial B_1)$ with the regularity constant depending on $n$ and $s$ only.

If $u(x_0)\geq 0$, then there exists a unit vector $e$ and universal constant $C>0$ such that in the set $\{x\in \partial B_1, e\cdot(x-x_0)\geq \frac{1}{2}|x-x_0| \}$, we have
\begin{equation}
\begin{aligned}
u(x)&\geq u(x_0)+\frac{1}{2}e\cdot(x-x_0)+O(|x-x_0|^{1+\frac{1-s}{2}})\\
&\geq \frac{1}{4}|x-x_0|-C|x-x_0|^{1+\frac{1-s}{2}}.
\end{aligned}
\end{equation}
Then there exists universal constant $\rho_0>0$ such that when $|x-x_0|<\rho_0$ and $e\cdot(x-x_0)\geq \frac{1}{2}|x-x_0|$,
\[
u(x)\geq \frac{1}{8}|x-x_0|.
\]
In the set $\Lambda=\{x\in \partial B_1, e\cdot(x-x_0)\geq \frac{1}{2}|x-x_0|, \rho_0/2<|x-x_0|<\rho_0\}$, we have
\[
u(x)\geq \frac{1}{16}\rho_0.
\] 
If $u(x_0)<0$, then there exists unit vector $e$, universal constants $C,\rho_0>0$ such that in the set $\Lambda=\{x\in \partial B_1, e\cdot(x-x_0)\leq -\frac{1}{2}|x-x_0|, \rho_0/2<|x-x_0|<\rho_0\}$,
\[
u(x)\leq u(x_0)-\frac{1}{4}|x-x_0|+C|x-x_0|^{1+\frac{1-s}{2}}\leq -\frac{1}{8}|x-x_0|\leq -\frac{1}{16}\rho_0,
\]

In both cases, we can thus write
\begin{equation}
\begin{aligned}
Cm\sqrt{\alpha m^{\frac{1-s}{n}}}\geq|E\Delta B_{r_0}|&\geq \int_{\partial B_1}{r_0^n|(1+u(x))^n-1|d\mathcal{H}^{n-1}(x)}\\
&\geq Cm\int_{\Lambda}{|u(x)|d\mathcal{H}^{n-1}(x)}\\
&\geq Cm |\Lambda|\frac{1}{16}\rho_0\\
&\geq Cm\rho_0^{n}
\end{aligned}
\end{equation}
%and
%\begin{equation}
%\begin{aligned}
%Cm\sqrt{\alpha m^{\frac{1-s}{n}}}\geq|E\Delta B_{r_0}|&\geq \int_{\partial B_1}{r_0^n|(1+u(x))^n-1|dH^{n-1}(x)}\\
%&\geq Cm\int_{\Lambda'}{|u(x)|d\mathcal{H}^{n-1}(x)}\\
%&\geq Cm |\Lambda'|\frac{1}{16}\rho_0\\
%&\geq Cm\rho_0^{n}.
%\end{aligned}
%\end{equation}
which implies:
\[
C\sqrt{\alpha m^{\frac{1-s}{n}}}=C\sqrt{\gamma}\geq C\rho_0^{n}
\]
and leads to a contradiction if 
$\gamma <\gamma_1''$ small constant (depending only on $n$ and $s$).
This completes the proof of \eqref{eq:gLinfty}.
\end{proof}

With Lemma \ref{Linfinitybound}, we can apply Theorem 2.1 in \cite{Fuglede} and obtain 
\begin{proposition}\label{Fuglede}
Let $E$ be a minimizer of \eqref{eq:min}
such that
\[
\partial E=\{r_0(1+u(x))x:x\in \partial B_1\} 
\]
for some function $u\in C^{1,\alpha}(\partial B_1)$. There exists a universal constant $\bar \gamma_1>0$ so that when $\gamma<\bar \gamma_1$,
\begin{equation}
D(E)\geq C(\|\nabla u\|^2_{L^{2}(\partial B_1)}+\| u\|^2_{L^{2}(\partial B_1)}).
\end{equation}

\end{proposition}

In order to get an upper bound on $D(E)$, we will show the following proposition:
\begin{proposition}\label{prop:1+snorm}
If $F\subset \R^n$ is an open set such that $|F|=m=|B_{r_0}|$ and
\[
\partial F=\{r_0(1+u(x))x:x\in \partial B_1\} 
\]
for some function $u\in C^{1,\frac{1}{2}(1-s)}(\partial B_1)$, then
\[
P_K(F)-P_K(B_{r_0})\leq C m^{\frac{n-s}{n}}\left([u]^2_{H^{\frac{1+s}{2}} (\partial B_1)}+\frac{1}{s(1-s)} \|u\|^2_{L^2(\partial B_1)}\right)
\]
for some constant $C$ depending only on $n$.
\end{proposition}

Before turning to the proof of this proposition, we will use Propositions  \ref{Fuglede}  and  \ref{prop:1+snorm} to prove Theorem \ref{thm:ball}:
\begin{proof}[Proof of Theorem \ref{thm:ball}]
When $\gamma\leq  \gamma_1'$, Theorem \ref{regularity} implies $\partial E\in C^{1,\frac{1}{2}(1-s)}$ and
\[
\partial E=\{r_0(1+u(x))x:x\in \partial B_1\}.
\] 
Furthermore, we have
\[
\J_{\alpha}(E)\leq\J_{\alpha}(B_{r_0}),
\]
and thus by Proposition \ref{prop:1+snorm},
\begin{align*}
D(E)=\frac{P(E)-P(B_{r_0})}{P(B_{r_0})} & \leq \alpha s(1-s)\frac{P_K(E)-P_K(B_{r_0})}{r_0^{n-1}P(B_1)}\\
&\leq C\alpha m^{\frac{1-s}{n}}(s(1-s)[u]^2_{H^{\frac{1+s}{2}}}+\|u\|^2_{L^2(\partial B_1)}).
\end{align*}
On the other hand,   Proposition \ref{Fuglede} gives the lower bound:
\[
D(E)\geq C(\|u\|_{L^2(\partial B_{1})}^2+[u]_{H^1(\partial B_{1})}^2)
\]
as long as $\gamma<\bar \gamma_1$.
Thus, we obtain
\[
C(\|u\|_{L^2(\partial B_{1})}^2+[u]_{H^1(\partial B_{1})}^2)\leq D(E)\leq C\alpha m^{\frac{1-s}{n}}(s(1-s)[u]^2_{H^{\frac{1+s}{2}}(\partial B_1)}+\|u\|^2_{L^2(\partial B_1)}).
\]
When $\gamma=\alpha m^{\frac{1-s}{n}}\leq \gamma_1$
with $\gamma_1$ small enough (but depending only on $n$ and $s$), this inequality, together with Sobolev embedding implies $u=0$ and so $E=B_{r_0}$.
\end{proof}

\begin{proof}[Proof of Proposition \ref{prop:1+snorm}]
The proof follows similar computations as in \cite{figalli}, although the fact that $P_K\neq P_s$ complicates things a bit.
First, we write $P_K(F)$ as
\begin{align*}
P_K(F)&=2\int_{F}{\int_{F^c} {K(x-y)dy}dx}\\
&=2 \int_{\partial B_1}{d\mathcal{H}^{n-1}(x)\int_{\partial B_1}{d\mathcal{H}^{n-1}(y)\int_{0}^{r_0(1+u(x))}{dr\int_{r_0(1+u(y))}^{\infty}{ f_{|x-y|}(r,\rho)d\rho}}}}\\
\end{align*}
with
\[
f_{|x-y|}(r,\rho)=r^{n-1}\rho^{n-1}K(\sqrt{(r-\rho)^2+r\rho|x-y|^2}).
\]
By using the identity  
\[
\int_0^b{\int_a^{\infty}}+\int_0^a{\int_b^{\infty}}=\int_a^b{\int_a^{b}}+\int_0^a{\int_a^{\infty}}+\int_0^b{\int_b^{\infty}},
\]
and the symmetric property of $f$, we have 
\begin{align}
P_K(F)&= \int_{\partial B_1}{d\mathcal{H}^{n-1}(x)\int_{\partial B_1}{d\mathcal{H}^{n-1}(y)\int_{r_0(1+u(x))}^{r_0(1+u(y))}{dr\int_{r_0(1+u(x))}^{r_0(1+u(y))}{ f_{|x-y|}(r,\rho)d\rho}}}}\nonumber \\
&\quad +2\int_{\partial B_1}{d\mathcal{H}^{n-1}(x)\int_{\partial B_1}{d\mathcal{H}^{n-1}(y)\int_{0}^{r_0(1+u(x))}{dr\int_{r_0(1+u(x))}^{\infty}{ f_{|x-y|}(r,\rho)d\rho}}}}\nonumber \\
&=I_1+I_2.\label{eq:I1I2}
\end{align}
For the first integral, we use the fact that for every $x,y\in \partial B_1$ we have
\[
f_{|x-y|}(r,\rho)\leq \frac{r^{n-1}\rho^{n-1}}{(r\rho|x-y|^2)^{\frac{n+s}{2}}}= \frac{r^{\frac{n-s}{2}-1}\rho^{\frac{n-s}{2}-1}}{|x-y|^{n+s}}
\]
since $0\leq K(z)\leq \frac{1}{|z|^{n+s}}$ and $K(|z|)$ is non-increasing. This implies
\begin{align*}
&\int_{r_0(1+u(x))}^{r_0(1+u(y))}{dr\int_{r_0(1+u(x))}^{r_0(1+u(y))}{ f_{|x-y|}(r,\rho)d\rho}}\\
&\qquad \leq \frac{1}{|x-y|^{n+s}}\int_{r_0(1+u(x))}^{r_0(1+u(y))}{r^{\frac{n-s}{2}-1}dr}\int_{r_0(1+u(x))}^{r_0(1+u(y))}{\rho^{\frac{n-s}{2}-1}d\rho}\\
&\qquad \leq \frac{4}{(n-s)^2}r_0^{n-s}\frac{((1+u(y))^{\frac{n+s}{2}}-(1+u(x))^{\frac{n+s}{2}})^2}{|x-y|^{n+s}}\\
&\qquad \leq Cr_0^{n-s}\frac{|u(x)-u(y)|^2}{|x-y|^{n+s}},
\end{align*}
and therefore,
\begin{equation}\label{eq:I1}
I_1\leq Cr_0^{n-s}\int_{\partial B_1}{\int_{\partial B_1}{\frac{|u(x)-u(y)|^2}{|x-y|^{n+s}}d\mathcal{H}^{n-1}(y)}d\mathcal{H}^{n-1}(x)}
=Cr_0^{n-s}[u]^2_{H^{\frac{1+s}{2}}(\partial B_1)}.
\end{equation}

For the second integral, fix $x\in \partial B_1$, and by the symmetric properties of $K$, the integral
\[
\int_{\partial B_1}{d\mathcal{H}^{n-1}(y)\int_{0}^{r_0(1+u(x))}{dr\int_{r_0(1+u(x))}^{\infty}{ f_{|x-y|}(r,\rho)d\rho}}}
\]
is a function of $u(x)$ only, so for every $\eta>0$, we define
\begin{equation}\label{psifunction}
\psi(\eta):=\int_{\partial B_1}{d\mathcal{H}^{n-1}(y)}{\int_{0}^{r_0\eta}{dr\int_{r_0\eta}^{\infty}{ f_{|x-y|}(r,\rho)d\rho}}}
\end{equation}
and write
\[
I_2=\int_{\partial B_1}{\psi(1+u(x))d\mathcal{H}^{n-1}(x)}.
\]
Note that \eqref{eq:I1I2} implies
\begin{equation}\label{P_J-psi}
P_K(B_{r_0\eta})=\int_{\partial B_1}{\psi(\eta)d\mathcal{H}^{n-1}(x)},\qquad \psi(\eta)=\frac{P_K(B_{r_0\eta})}{P(B_1)}.
\end{equation}

Together with \eqref{eq:I1}, \eqref{eq:I1I2} now implies
\[
P_K(F)=I_1+I_2\leq Cr_0^{n-s}[u]^2_{H^{\frac{1+s}{2}}(\partial B_1)}+\int_{\partial B_1}{\psi(1+u(x))d\mathcal{H}^{n-1}(x)}
\]
and using \eqref{psifunction}, we deduce
\begin{equation}\label{eq:PKFPKB}
P_K(F)-P_K(B_{r_0})\leq Cr_0^{n-s}[u]^2_{H^{\frac{1+s}{2}}(\partial B_1)}+\int_{\partial B_1}{(\psi(1+u(x))-\psi(1))d\mathcal{H}^{n-1}(x)}.
\end{equation}
To complete the proof, we use the following technical lemma, which is proved in Appendix \ref{sec:psi}:
\begin{lemma}\label{lem:psi}
There exists $\beta \in \mathbb{R}$ with $|\beta|\leq \frac{C}{s(1-s)}r_0^{n-s}$ such that
\[
|\psi(1+\lambda)-\psi(1)-\beta \lambda|\leq \frac{C}{s(1-s)}r_0^{n-s}\lambda^2\]
for all $|\lambda|<1/2$.
\end{lemma}
With $\beta$ as in Lemma \ref{lem:psi}, we get
\begin{equation}\label{claim}
\begin{aligned}
\int_{\partial B_1}{|\psi(1+u(x))-\psi(1)-\beta u(x)|d\mathcal{H}^{n-1}(x)}\leq \frac{C}{s(1-s)}r_0^{n-s}\int_{\partial B_1}{u^2(x)d\mathcal{H}^{n-1}(x)}.
\end{aligned}
\end{equation}
Furthermore,  the volume constraint $\int_{\partial B_1}{((1+u(x))^n-1)d\mathcal{H}^{n-1}(x)}=0$ implies 
\[
\left|\int_{\partial B_1}{\beta u(x)d\mathcal{H}^{n-1}(x)}\right| \leq C|\beta|\int_{\partial B_1}{ u^2(x)d\mathcal{H}^{n-1}(x)}.
\]
We thus have
\begin{equation*}
\begin{aligned}
\int_{\partial B_1}{(\psi(1+u(x))-\psi(1))d\mathcal{H}^{n-1}(x)}&\leq \int_{\partial B_1}{|\psi(1+u(x))-\psi(1)-\beta u(x)|d\mathcal{H}^{n-1}(x)}\\
&\qquad +\left|\int_{\partial B_1}{\beta u(x)d\mathcal{H}^{n-1}(x)}\right|\\
&\leq \frac{C}{s(1-s)}r_0^{n-s}\int_{\partial B_1}{u^2(x)d\mathcal{H}^{n-1}(x)}.
\end{aligned}
\end{equation*}
The result now follows from \eqref{eq:PKFPKB}.
%and thus it implies
%\[
%P_K(F)-P_K(B_{r_0})\leq Cr_0^{n-s}([u]^2_{H^{\frac{1+s}{2}}(\partial B_1)}+\|u\|^2_{L^2(\partial B_1)}).
%\]
\end{proof}

\section{The ball is not a minimizers for large $\alpha m^{\frac{1-s}{n}}$} \label{sec:nonball}
\subsection{Proof of Theorem \ref{thm:nonball}}
Theorem \ref{thm:nonball} states in particular that the ball is not a global minimizer for large $\alpha m^{\frac{1-s}{n}}$, but also gives an upper bound on the width the minimizer. 
To prove the result, we assume, by contradiction that there exists $x_0\in\mathbb{R}^n$ such that $B_{\rho/2}(x_0)\subset B_{\rho}(x_0)\subset E$ (and so $|B_{\rho}(x_0)\setminus E|=0$) for some $\rho\geq 4\rho_0$.
We then construct a competitor $F$ by removing the ball $B_{\rho/2}(x_0)$ from $E$ and placing it "at infinity".
More precisely, given $y_0$ such that $|x_0-y_0|=R\gg1$, and $B_{\rho/2}(y_0)\cap E=\emptyset$, 
we consider the set
\[
F_R=\big( E\setminus B_{\rho/2}(x_0)\big)\cup B_{\rho/2}(y_0)
\]
Using \eqref{eq:PKAB}, we get
$$
P_K(F_R) = P_K(E\setminus B_{\rho/2}(x_0)) + P_K(B_{\rho/2}(y_0)) - 4 \int_{\R^n} \int_{|R^n} K(x-y) \chi_{E\setminus B_{\rho/2}(x_0)} (x) \chi_{B_{\rho/2}(y_0)}(y)\, dx\,dy
$$
and 
$$
P_K(E) = P_K(E\setminus B_{\rho/2}(x_0)) + P_K(B_{\rho/2}(x_0)) - 4 \int_{\R^n} \int_{|R^n} K(x-y) \chi_{E\setminus B_{\rho/2}(x_0)} (x) \chi_{B_{\rho/2}(x_0)}(y)\, dx\,dy
$$
and therefore (using the fact that $K(x)\to0 $ as $|x|\to \infty$):
\begin{align*}
P_K(F_R) - P_K(E) & =  4 \int_{\R^n} \int_{|R^n} K(x-y) \chi_{E\setminus B_{\rho/2}(x_0)} (x) \chi_{B_{\rho/2}(x_0)}(y)\, dx\,dy \\
& \quad - 4 \int_{\R^n} \int_{|R^n} K(x-y) \chi_{E\setminus B_{\rho/2}(x_0)} (x) \chi_{B_{\rho/2}(y_0)}(y)\, dx\,dy \\
& \geq   4 \int_{\R^n} \int_{|R^n} K(x-y) \chi_{E\setminus B_{\rho/2}(x_0)} (x) \chi_{B_{\rho/2}(x_0)}(y)\, dx\,dy   + o(1)\qquad  \mbox{ as } R\to \infty.
\end{align*}
Since $|F_R|=|E|=m$, we must have $\J_{\alpha}(F_R)\geq \J_{\alpha}(E)$ and so
\begin{align*}
C\rho^{n-1} \geq 2 P(B_{\rho/2})&\geq \alpha s(1-s)   (P_K(F_R)-P_K(E))\\
&\geq 4  \alpha s(1-s) \int_{B_{\rho}(x)\setminus B_{\frac{3}{4}\rho}(x)}{\int_{B_{\frac{1}{4}\rho}(x)}{K(\xi-\eta)d\eta}d\xi} + o(1)\\
&\geq \alpha s(1-s) C\rho^{2n}\frac{1}{\rho^{n+s}}+ o(1)
\end{align*}
where we used the lower bound \eqref{eq:assK2} and the fact that $\rho\geq 4\rho_0$.
Passing to the limit $R\to \infty$, this implies that
\[
\rho\leq (C\alpha s (1-s))^{\frac{-1}{1-s}}
\]
and the result follows.
\medskip

In particular, if the ball of volume $m$ is a minimizer, we must have
\[
r_0=\left(\frac{m}{|B_1|}\right)^{1/n}\leq\max \{ 4 \rho_0,(C\alpha s (1-s))^{\frac{-1}{1-s}} \} 
%> \tilde{\gamma}_3\alpha^{\frac{-1}{1-s}},
\]
which is equivalent to
$$
\alpha m^{\frac{1-s}{n}} \leq \max \{\gamma_2, c \alpha \rho_0^{1-s}\} 
$$
for some $\gamma_2$ and $c$ depending on $s$ and $n$.
\medskip

\subsection{Proof of Theorem \ref{thm:beta1beta2}}
%Finally, we further characterize the role of the ball.
We note that $E$ is a minimizer $\F_{s,\alpha}$ with the constraint $|E|=m$, if and only if $E_m=\frac{1}{m^{1/n}}E$ is a minimizer of energy $\F_{s,\beta}$ under the constraint $|F|=|B_1|$ with 
\begin{equation}\label{eq:beta}
\beta:=\alpha (m/|B_1|) ^{\frac{1-s}{n}}
\end{equation}
The first part of the theorem relies on the following lemma:
\begin{lemma}\label{globalintervalbeta}
If $B_1$ is the global minimizer of $\F_{s,\beta}(F)$ with constraint $|F|=|B_1|$ for $\beta=\beta_1$, then it is the global minimizer of $\F_{s,\beta}(F)$ for all $\beta<\beta_1$.
\end{lemma}
\begin{proof}
Let $E$ be such that $|E|=|B_1|$ and 
\[
P(E)-\beta s(1-s)P_s(E)\leq P(B_1)-\beta s(1-s)P_s(B_1)
\]
for some $\beta<\beta_1$. This implies
\begin{equation}\label{beta0inequality}
P(E)-\beta_1 s(1-s) P_s(E)+(\beta_1-\beta)s(1-s)P_s(E)\leq P(B_1)-\beta_1 s(1-s)P_s(B_1)+(\beta_1-\beta)s(1-s)P_s(B_1).
\end{equation}

Since $B_1$ is the global minimizer of $\F_{\beta_1}$, we have
\[
P(E)-\beta_1 s(1-s)P_s(E)\geq P(B_1)-\beta_1 s(1-s) P_s(B_1),
\]
so \eqref{beta0inequality} (and the fact that $\beta_1-\beta>0$) yields
\[
P_s(E)\leq P_s(B_1).
\]
The isoperimetric inequality for the fractional perimeter thus implies that $E=B_1$ and the lemma follows.
\end{proof}

Lemma \ref{globalintervalbeta}, together with Theorems \ref{thm:ball} and \ref{thm:nonball} imply that there exists $\beta_1^*\in(0,\infty)$ such that
$$\{\beta \, ;\, B_1  \mbox{ is a global minimizer of }\F_{s,\beta}\} = [0,\beta_1^*]$$ 
and thus prove the first part of Theorem \ref{thm:beta1beta2}.
\medskip
\medskip

The proof of the second part of the Theorem \ref{thm:beta1beta2} (local minimizers) relies on an analysis of the first and second variations of the energy $\F_{s,\beta}$.
Following \cite{figalli}, given a vector field $X\in C_c^{\infty}(\R^n;\R^n)$ we denote by $\{\Phi_t\}_{t\in \R}$ the flow induced by $X$ (solution of $\partial _t\Phi_t(x) = X(\Phi_t(x))$, $\Phi_0(x)=x$) and we then define $E_t := \Phi_t(E)$.
We say that
$X$ induces a volume-preserving flow on $E$ if $|E_t| =|E|$ for small $t$. This implies in particular
\begin{equation}\label{eq:vol}
\frac{d}{dt}|E_t| _{|_{t=0}} = \int_{\partial E} \zeta\, d\H^{n-1} = 0,  \qquad \frac{d^2}{dt^2}|E_t| _{|_{t=0}} = \int_{\partial E} (\div X)\, \zeta\, d\H^{n-1} = 0
\end{equation}
where we denoted 
$$ \zeta  := X\cdot \nu$$
with $\nu$ normal unit vector to $\partial E$.
With these notations, we have the following 
 formula for the first and second variation of the perimeters along a volume-preserving flow  (see for example Section 6 in \cite{figalli}):
\begin{align*}
 \delta P(E)[X]& =  \int_{\partial E } H_{\partial E}\,  \zeta\, d\H^{n-1}\\
 \delta^2P(E)[X] & =\int_{\partial E} |\nabla_{\tau}\zeta|^2 - c_{\partial E}^2\zeta^2d\mathcal{H}^{n-1}  \\
& \qquad+ \int_{\partial E} H_{\partial E} \left( (\div X) \zeta - \div_\tau (\zeta X_\tau)\right) d\H^{n-1}
\end{align*}
We used here the following classical notations: $H_{\partial E}$ denotes the scalar mean-curvature of $\partial E$, $c^2_{\partial E}$ denotes the sum of the squares of the principal curvatures of $\partial E$, $X_\tau=X-\zeta \nu$ is the tangential projection of $X$ along $\partial E$ and $\nabla_\tau$, $\div_\tau$ denotes the tangential gradient and divergence operators.

When $E=B_1$, all curvatures are constant and using \eqref{eq:vol}, we deduce:
\begin{align}
 \delta P(B_1)[X]& =  0\label{eq:var1P} \\
 \delta^2P(B_1)[X] & =\int_{\partial B_1}{|\nabla_{\tau}\zeta|^2d\mathcal{H}^{n-1}}-\int_{\partial B_1}{c_{\partial B_1}^2\zeta^2d\mathcal{H}^{n-1}}\label{eq:var2P} 
\end{align}

Similar formula are derived in \cite{figalli} for nonlocal perimeters. In particular we can write:
\begin{align*}
\delta P_s(E)[X] & = \int_{\partial E } H_{s,\partial E}\,  \zeta\, d\H^{n-1}\\
\delta^2 P_s(E)[X] & = \int\!\!\!\int_{\partial E\times\partial E} \frac{|\zeta(x)-\zeta(y)|^2}{|x-y|}^{n+s}\, d\H^{n-1}_x d\H^{n-1}_y- \int_{\partial E} c_{s,\partial E}^2\zeta^2d\mathcal{H}^{n-1}  \\
& \qquad+ \int_{\partial E} H_{s,\partial E} \left( (\div X) \zeta - \div_\tau (\zeta X_\tau)\right) d\H^{n-1}
\end{align*}
where $H_{s,\partial E}$ denotes the $s-$curvature defined by
$$ H_{s,\partial E}(x): = \mathrm{p.v.} \int_{\R^n} \frac{\chi_{E^c}(y)-\chi_E(y)}{|x-y|^{n+s}} dy$$
and we set
$$ c_{s,\partial E}^2(x) := \int_{\partial E} \frac{|\nu(x)-\nu(y)|^2} {|x-y|^{n+s}} d\H^{n-1}_y.$$

When $E=B_1$, using \eqref{eq:vol}, we get:
\begin{align}
& \delta P_s(B_1)[X]=0 \label{eq:var1Ps}  \\
& \delta^2 P_s(B_1)[X] = \int_{\partial B_1}{\int_{\partial B_1}{\frac{|\zeta(x)-\zeta(y)|^2}{|x-y|^{n+s}}d\mathcal{H}^{n-1}(x)}d\mathcal{H}^{n-1}(y)}-\int_{\partial B_1}{c_{s,\partial B_1}^2\zeta^2d\mathcal{H}^{n-1}}\label{eq:var2Ps} 
\end{align}

Using  \eqref{eq:var1P}  and \eqref{eq:var1Ps},   we see that the unit ball $B_1$ is always a 
critical point of the energy $\F_{s,\beta}$ for volume preserving variation, since
\[
\delta \F_{s,\beta}(B_1)[X]=\delta P(B_1)[X]-\beta s(1-s)\delta P_s(B_1)[X].
\]
Furthermore, the second variation formula yield:
 \begin{equation}\label{secondvariation}
 \begin{aligned}
\delta^2 \F _{\beta}(B_1)[X]&=\int_{\partial B_1}{|\nabla_{\tau}\zeta|^2d\mathcal{H}^{n-1}}-c_{\partial B_1}^2 \int_{\partial B_1}{\zeta^2d\mathcal{H}^{n-1}}\\
&-\beta s(1-s)\left(\int\!\!\!\int_{\partial B_1\times \partial B_1} \frac{|\zeta(x)-\zeta(y)|^2}{|x-y|^{n+s}} d\H^{n-1}_xd\H^{n-1}_y -c_{s,\partial B_1}^2\int_{\partial B_1}{\zeta^2d\mathcal{H}^{n-1}}\right)
\end{aligned}
\end{equation}

We will say that $B_1$ is a volume-constrained stable set for $\F_{s,\beta}$ if $\delta^2 \F_{s,\beta}(B_1)[X]\geq 0$ for every $X$ inducing a volume-preserving flow on $B_1$.
We then have the following result:
\begin{proposition}\label{prop:stable}
The unit ball $B_1$ is a volume-constrained stable set for $\F_{s,\beta}$ if and only if $\beta \in [0,\beta_2^*]$ where
\begin{equation}\label{eq:beta2}
\beta_2^*=\frac{n+1}{s(n+s)}\frac{P(B_1)}{s(1-s)P_s(B_1)}.
\end{equation}
Moreover, as $s\to 1$, $\beta_2^*\to \frac{1}{2n\omega_n}$ and as $s\to 0$, $\beta_2\to \infty$.
\end{proposition}
Using this Proposition and proceeding as in \cite{figalli}, we can then prove:
\begin{theorem}
With $\beta_2^*$ defined as in \eqref{eq:beta2}, we have:
If $\beta< \beta_2^*$ then $B_1$ is a local volume-constrained  minimizer of $\F_{s,\beta}$. 
If $\beta> \beta_2^*$ then $B_1$ is not a local volume-constrained  minimizer of $\F_{s,\beta}$. 
\end{theorem}
Note that the definition of $\beta$, \eqref{eq:beta}, then yields
$$ \gamma^*_2 = \beta_2^* |B_1|^{\frac{1-s}{n}}.$$
\begin{proof}[Proof of Proposition \ref{prop:stable}]
We introduce the set
\[
\widetilde{H}^1(\partial B_1)=\left\{ \zeta\in H^1(\partial B_1), \int_{\partial B_1}{\zeta d\mathcal{H}^{n-1}}=0 \right\},
\]
and we denote
$$
[\zeta]_{H^1(\partial B_1}^2 = \int_{\partial B_1}{|\nabla_{\tau}\zeta|^2d\mathcal{H}^{n-1}}, \quad [\zeta]^2_{H^{\frac{1+s}{2}}(\partial B_1)} =\int\!\!\!\int_{\partial B_1\times \partial B_1} \frac{|\zeta(x)-\zeta(y)|^2}{|x-y|^{n+s}} d\H^{n-1}_xd\H^{n-1}_y.
$$
Using \eqref{secondvariation}, we see that the condition $\delta^2 \F_{s,\beta}(B_1)[X]\geq 0$ for all $X$ is equivalent to
 \begin{equation}\label{beta2}
\beta
\leq \beta_2^*:=\frac{1}{s(1-s)}\inf_{\zeta \in \tilde{H}^1(\partial B_1)}{\frac{[\zeta]_{H^1(\partial B_1}^2-c_{\partial B_1}^2\|\zeta\|_{L^2(\partial B_1)}^2}{[\zeta]^2_{H^{\frac{1+s}{2}}(\partial B_1)}-c_{s,\partial B_1}^2\|\zeta\|_{L^2(\partial B_1)}^2}}.
\end{equation}
%Then when $\beta>\beta_2$, $\delta^2J_{\beta}(B_1)[X]<0$ for some $X\in C_C^{\infty}(\R^n;\R^n)$ and $B_1$ is not a local minimizer; and when $\beta<\beta_2$, $\delta^2J_{\beta}(B_1)[X]\> 0$ for all $X$ and $B_1$ is a local minimizer. 
In order to find the value of $\beta$ given by \eqref{eq:beta2}, we follow \cite{figalli} again by introducing an orthogonal basis for $L^2(\partial B_1)$ given by 
$\left\{ \{ Y^i_k\}_{i=1}^{d(k)} \right\}_{k\in \N}$ (where $ \{ Y^i_k\}_{i=1}^{d(k)}$ is an orthogonal basis of where  the fine dimensional subspace of spherical harmonics of degree $k$).
Denoting $a_k^i(\zeta) = \int_{\partial  B_1} \zeta Y_k^i\, d\H^{n-1}$, we then have:
\begin{align*} 
\|\zeta\|_{L^2(\partial B_1)}^2 & = \sum_{k=0}^\infty \sum_{i=1}^{d(k)} a_k^i(\zeta)^2 \\
[\zeta]_{H^1(\partial B_1}^2& = \sum_{k=0}^\infty \sum_{i=1}^{d(k)} \lambda_k^1 a_k^i(\zeta)^2 \\
[\zeta]^2_{H^{\frac{1+s}{2}}(\partial B_1)} & = \sum_{k=0}^\infty \sum_{i=1}^{d(k)} \lambda_k^s a_k^i(\zeta)^2 
\end{align*}
which 
implies 
$$\beta_2^*=\frac{1}{s(1-s)}\inf_{k\geq 2}{\frac{\lambda_k^1-\lambda_1^1}{\lambda_k^s-\lambda_1^s}}.$$
Lemma \ref{lemmabeta2} below states that this minimum is reached when $k=2$ and so
\[
\beta_2^*=\frac{1}{s(1-s)}\frac{\lambda_2^1-\lambda_1^1}{\lambda_2^s-\lambda_1^s}.
\]
Finally, using formula (2.11), (2.12) and Proposition 7.2 in \cite{figalli}, we conclude:
\[
\beta_2^*=\frac{n+1}{s(n+s)}\frac{P(B_1)}{s(1-s)P_s(B_1)}.
\]
\end{proof}

\begin{lemma}\label{lemmabeta2}
For all $n$ and $s\in(0,1)$, we have:
\[
\inf_{k\geq 2}{\frac{\lambda_k^1-\lambda_1^1}{\lambda_k^s-\lambda_1^s}}=\frac{\lambda_2^1-\lambda_1^1}{\lambda_2^s-\lambda_1^s}.
\]
\end{lemma}
\begin{proof}
To prove for every $k\geq 2$,
\[
\frac{\lambda_k^1-\lambda_1^1}{\lambda_k^s-\lambda_1^s}\geq \frac{\lambda_2^1-\lambda_1^1}{\lambda_2^s-\lambda_1^s},
\]
it is equivalent to prove
\[
A_k=\frac{k(k+n-2)-n+1}{n+1}\frac{\frac{1+\frac{n+s}{2}}{1+\frac{n-2-s}{2}}-1}{\frac{\prod_{j=1}^{k-1}{(j+\frac{n+s}{2})}}{\prod_{j=1}^{k-1}{(j+\frac{n-2-s}{2})}}-1}\geq 1
\]
for all $k\geq 2$.
We prove by deduction. When $k=2$, $A_2=1$. Assume $A_k\geq 1$ which implies
 \begin{equation}\label{Akexp}
\frac{\prod_{j=1}^{k-1}{(j+\frac{n+s}{2})}}{\prod_{j=1}^{k-1}{(j+\frac{n-2-s}{2})}}\leq \frac{(k-1)(n+k-1)}{n+1}\frac{1+s}{1+\frac{n-2-s}{2}}+1,
\end{equation}
and we aim to prove $A_{k+1}\geq 1$. With \eqref{Akexp} and 
\[
\frac{\prod_{j=1}^{k}{(j+\frac{n+s}{2})}}{\prod_{j=1}^{k}{(j+\frac{n-2-s}{2})}}=\frac{\prod_{j=1}^{k-1}{(j+\frac{n+s}{2})}}{\prod_{j=1}^{k-1}{(j+\frac{n-2-s}{2})}}\frac{k+\frac{n+s}{2}}{k+\frac{n-2-s}{2}},
\]
$A_{k+1}\geq 1$ is true if we can show
\[
\left(\frac{(k-1)(n+k-1)}{n+1}\frac{1+s}{1+\frac{n-2-s}{2}}+1\right)\frac{k+\frac{n+s}{2}}{k+\frac{n-2-s}{2}}\leq \frac{k(n+k)}{n+1}\frac{1+s}{1+\frac{n-2-s}{2}}+1.
\]
After simplification, it is equivalent to 
\[
(n+k)(k-1)(1-s)\geq 0
\]
and this implies $A_{k+1}\geq 1$ if $A_k\geq 1$. By induction we prove $A_k\geq 1$ for all $k\geq 2$.
\end{proof}

In order to complete the proof of Theorem \ref{thm:beta1beta2}, it only remains to show that $\beta_0<\beta_2$. This follows from the following Lemma:
\begin{lemma}\label{lem:beta}
For all $n$ and $s$, we have
$$ \beta_1^* \leq \frac{2^{\frac{1}{n}}-1}{2^{\frac{s}{n}}-1}\frac{P(B_1)}{s(1-s)P_s(B_1)} < \beta_2^*.$$
\end{lemma}
\begin{proof}[Proof of Lemma \ref{lem:beta}]
To prove the lemma, we will show that for $\beta>\bar \beta= \frac{2^{\frac{1}{n}}-1}{2^{\frac{s}{n}}-1}\frac{P(B_1)}{s(1-s)P_s(B_1)} $,  $B_1$ is not a global minimizer of $\F_{s,\beta}$. 
This is done by showing that the set made of two balls (each with volume $|B_1|/2$) far away from each other has a lower energy than $B_1$ when $\beta>\bar \beta$.

Let $E_R=B_{2^{-\frac{1}{n}}}(0)\cup B_{2^{-\frac{1}{n}}}(Re_1)$. Then $|E_R|=|B_1|$ and  
\[
\F_{s,\beta}(E_R)=2^{\frac{1}{n}}P(B_1)-\beta s(1-s)2^{\frac{s}{n}}P_s(B_1) + \mathcal O(1/R^{n+s}).
\] 
It is thus a simple exercise to show that $\F_{s,\beta}(E_R)<\F_{s,\beta}(B_1)$ if $\beta>\bar \beta$ and $R$ is large enough.

It remains to prove that $\bar \beta<\beta_2^*$. It is equivalent to show for $s\in(0,1)$ and $n\geq 2$,
\[
f(s):=\frac{2^{s/n}-1}{2^{1/n}-1}-\frac{s(n+s)}{n+1}>0.
\]
We can see 
 \begin{equation*}
 \begin{aligned}
f''(s)=\frac{(\ln{2})^22^{s/n}}{n^2(2^{1/n}-1)}-\frac{2}{n+1}&<\frac{(\ln{2})^22^{1/n}}{n^2(2^{1/n}-1)}-\frac{2}{n+1}\\
&\leq \frac{(\ln{2})^2\sqrt{2}}{2(2^{1/2}-1)n}-\frac{2}{n+1}\\
&<\frac{1}{n}-\frac{2}{n+1}<0,
\end{aligned}
\end{equation*}
in which we use the fact that
$\inf_{n\geq 2}{n(2^{1/n}-1)}=2(2^{1/2}-1)$.
Since $f(s)$ is defined on $[0,1]$ with $f(0)=f(1)=0$ and $f''(s)<0$, we can prove $f(s)>0$ for all $n\geq 2$ and $s\in(0,1)$ and thus $\beta_1^*\leq \bar \beta< \beta_2^*$. 
\end{proof}

\section{Proof of Theorem \ref{thm:G}}\label{sec:G}
In order to prove Theorem \ref{thm:G},
we first add a confinement term to the energy functional $\J_\alpha$:
$$
\J_{\alpha,\beta} (E) = P(E) - \alpha s(1-s)P_K(E) + \beta \int_{\R^n} |x|\chi_E(x)\, dx.
$$
We then have the following result:
\begin{proposition}\label{prop:minbeta}
For all $\alpha>0$, $\beta>0$ and $m>0$, there exists a minimizer $E$ of $\J_{\alpha,\beta}$ with the constraint $|E|=m$. This minimizer satisfies
\begin{equation}\label{eq:minbound}
P(E) + \beta \int_{\R^n} |x| \chi_E(x)\, dx \leq C((1+(n\omega_n \gamma)^\frac{1}{1-s}) m^{\frac{n-1}{n}} + \beta m^\frac{n+1}{n}), \qquad \gamma =  \alpha\, m^{\frac{1-s}{n}}.
\end{equation}
Furthermore, with $\gamma_1$ given by Theorem \ref{thm:ball}, we have that if 
$$
\alpha\, m^{\frac{1-s}{n}}\leq  \gamma_1
$$
then $E$ is the ball of volume $m$.
\end{proposition}

\begin{proof}
For every $m,\alpha,\beta>0$, using \eqref{eq:KP}, we can show that $\J_{\alpha,\beta}(E)$ has a lower bound
\[
\J_{\alpha,\beta}(E)\geq \frac{1}{2}P(E)+\beta\int{|x|\chi_{E}(x)dx}-(n\omega_n\alpha)^{\frac{1}{1-s}}m.
\]
In particular, there exists a minimizing sequence $\{E_j\}_{j\in \mathbb{N}}$ with $|E_j|=|B_{r_0}|=m$ and we can always assume that 
\[
\J_{\alpha,\beta}(E_j)\leq \J_{\alpha,\beta}(B_{r_0}) \qquad \forall j\in \mathbb{N}. 
\]
This implies
\begin{equation*}
\begin{aligned}
\frac{1}{2} P(E_j)+\beta\int{|x|\chi_{E_j}(x)dx}&\leq P(B_{r_0})+\beta\int{|x|\chi_{B_{r_0}}(x)dx}+(n\omega_n\alpha)^{\frac{1}{1-s}}m\\
&\leq Cm^{\frac{n-1}{n}}+\beta Cm^{\frac{n+1}{n}}+(n\omega_n\alpha)^{\frac{1}{1-s}}m\\
& \leq C(1+(n\omega_n \gamma)^{\frac{1}{1-s}})^sm^{\frac{n-1}{n}}+C\beta m^{\frac{n+1}{n}}=C(m,\alpha,\beta)
%&\leq Cm^{\frac{n-1}{n}}+\beta Cm^{\frac{n+1}{n}}+\frac{1}{2}P(E)+\frac{1}{2}(n\omega_n \gamma)^{\frac{s}{1-s}}m^{\frac{n-1}{n}},
\end{aligned}
\end{equation*}
%By \eqref{eq:PEbound}, we have
%\[
%P(E_j)\leq C(1+(n\omega_n \gamma)^{\frac{1}{1-s}})^sm^{\frac{n-1}{n}}+C\beta m^{\frac{n+1}{n}}=C(m,\alpha,\beta)
%\]
for all $j\in \N$. So $\chi_{E_j}$ is uniformly bounded in $BV(\R^n)$ and thus relatively compact in $L^1_{loc}(\R^n)$.

%Assume $E$ satisfies $\J_{\alpha,\beta}(E)\leq \J_{\alpha,\beta}(B_{r_0})$ with $|E|=|B_{r_0}|=m$. Using \eqref{eq:KP}, we obtain
%\begin{equation*}
%\begin{aligned}
%P(E)+\beta\int{|x|\chi_E(x)dx}&\leq P(B_{r_0})+\beta\int{|x|\chi_{B_{r_0}}(x)dx}+\alpha s(1-s)P_K(E)\\
%&\leq Cm^{\frac{n-1}{n}}+\beta Cm^{\frac{n+1}{n}}+\alpha c_nP(E)^s|E|^{1-s}\\
%&\leq Cm^{\frac{n-1}{n}}+\beta Cm^{\frac{n+1}{n}}+\frac{1}{2}P(E)+\frac{1}{2}(n\omega_n \gamma)^{\frac{s}{1-s}}m^{\frac{n-1}{n}},
%\end{aligned}
%\end{equation*}
%with $\gamma= \alpha m^{\frac{1-s}{n}}$. This implies
%\begin{equation}\label{eq:PEbound}
%\begin{aligned}
%P(E)+\beta\int{|x|\chi_E(x)dx}&\leq P(E)+2\beta\int{|x|\chi_E(x)dx}\\
%&\leq C(1+(n\omega_n \gamma)^{\frac{1}{1-s}})^sm^{\frac{n-1}{n}}+C\beta m^{\frac{n+1}{n}}\\
%&=C(m,\alpha,\beta).
%\end{aligned}
%\end{equation}

By a diagonal extraction argument, we deduce the existence of a subsequence $\chi_{E_{j_k}}$ and a set $E$ such that $\chi_{E_{j_k}}$ converges to $\chi_{E}$ weakly in $L^1(\R^n)$ and strongly in $L^1(B_R)$ for all $R>0$.
The lower semicontinuity of the perimeter implies
$$P(E)\leq \liminf_{k\to \infty}{P(E_{j_k})},$$
and Lemma \ref{lem:CVPK} gives
%and the boundedness in $BV(\R^n)$ implies (see Lemma
$$
P_K(E)=\lim_{j\to \infty}{P_K(E_{k_j})}.
$$
%Given any $R>0$, since $BV(B_R)\subset L^1(B_R)$, there exists a subsequence $\{\chi_{E_{j_k}}\}$ which is still denoted as $\{\chi_{E_j}\}$ such that
%\[
%\chi_{E_j}\to u
%\]
%a.e., strongly in $L^1(B_R)$, and weakly in $BV(B_R)$. We have $u=\chi_E$ for $E\subset \R^n$ since $u(x)=\lim_{j\to\infty}{\chi_{E_j}}$ a.e. and $\chi_{E_j}\in \{0,1\}$. We have $\chi_E\in BV(B_R)$ since for every $g\in C_c^1(B_R,\R^n)$ and $|g|\leq 1$,
%\[
%\int_{B_R}{\chi_E div g}=\lim_{j\to \infty}{\int_{B_R}{\chi_{E_k} div g}}\leq \liminf_{j\to \infty}{P(E_k,B_R)},
%\]
%which implies
%\begin{equation}\label{eq:Pliminf}
%P(E,B_R)=\sup{\{\int_{B_R}{\chi_E div g}, g\in C_c^1(B_R,\R^n), |g|\leq 1\}}\leq \liminf_{j\to \infty}{P(E_k,B_R)}.
%\end{equation}
%We also have
%\[
%P_K(E,B_R)=\lim_{j\to \infty}{P_K(E_j,B_R)},
%\]
%and
Finally, the strong $L^1$ convergence implies
\[
\int_{B_R}{|x|\chi_E(x)dx}=\lim_{j\to \infty}{\int_{B_R}{|x|\chi_{E_j}dx}},
\]
and using the fact that
\[
R|E_j\cap B_R^c|\leq \int_{\R^n}{|x|\chi_{E_j}(x)dx}   \leq C(m,\alpha,\beta)
\]
we obtain
\[
m-\frac{1}{R}C(m,\alpha,\beta)\leq|E_j\cap B_R|\leq m.
\]
and therefore
%Then we can use the diagnol extraction precedure to obtain a subsequence of $\{\chi_{E_j}\}$ denoted as $\{\chi_{E_k}\}$ such that for every $k\in \N$, $\chi _{E_k}\to \chi_{E}$ a.e. and strongly in $L^1(B_k)$, weakly in $BV(B_k)$ with 
$|E|=m$. 
%Using the same argument in \eqref{eq:Pliminf} we obtain
%\[
%P(E)\leq \liminf_{k\to \infty}{P(E_k)}.
%\]
Altogether, we deduce that $E$ is a minimizer of $\J_{\alpha,\beta}$ with $|E|=m$.
\end{proof}

Next, for $N>0$, we introduce
$$ \G_{\alpha,\beta} (\mathbb E)  = \sum_{i=1}^N \J_{\alpha,\beta} (E_i), \quad \mathbb E = \{E_1,\dots , E_N\}$$
and $|\mathbb E| = \sum_{i=1}^N|E_i|$. 
Proceeding as in Proposition \ref{prop:minbeta}, we can show that for all $\alpha>0$, $\beta>0$, $m>0$, and $N$ the minimization problem
\begin{equation}\label{eq:minGNR}
\inf \left\{ \G_{\alpha,\beta} (\mathbb E) \, ;\, \mathbb E = \{E_i\}_{i=1 ,\dots, N} ,\;  |\mathbb E| = m\right\} 
\end{equation}
has a minimizer.

Note that if $|E_{i_0}|=0$ for some $i_0$, then that component of $\mathbb E$ does not contribute to the energy and we can set $E_{i_0}=\emptyset$. In the sequel, we will always assume that either $E_{i}=\emptyset$ or $|E_i|>0$.

We then prove:
\begin{lemma}
There exists $m_0(\alpha)>0$ and $N_0(\alpha)$ such that if $R\geq m_0^{1/n}$ and $N\geq N_0$, given $\mathbb E = \{E_i\}_{i=1 ,\dots, N}$ a global minimizer of \eqref{eq:minGNR}, there are at most $N_0$ components $E_i$ of $\mathbb E$ which are not the empty set.
\end{lemma}
\begin{proof}
We fix $m_0$ such that for all $m\leq m_0$ the minimizer of $\J_{\alpha,\beta}$ with constraint $|E|=m$ is the ball and such that the energy of two balls of mass $2m$ in total is bigger than that of one ball of mass $2m$.
It is then clear that $\mathbb E$ can have at most $1$ component with Lebesgue measure smaller than $m_0$ and therefore we can take $N_0 = \frac{m}{m_0}+1$.
\end{proof}

From now on, we fix $N\geq N_0+1$ so that at least one component of $\mathbb E$ has zero Lebesgue measure (and since we can remove such components without changing the energy, we can assume that one of the component is the empty set). 
We then have
\begin{lemma}
Let $\mathbb E = \{E_i\}_{i=1 ,\dots, N}$ be a  global minimizer of \eqref{eq:minGNR}. Then each component $E_i$ is indecomposable.
\end{lemma}
\begin{proof}
If one component $E_{i_0}$ is decomposable, that it can be written as $ E_{i_0} = F_1\cup F_2$ with $P( E_{i_0})  = P(F_1)+P( F_2)$ and $|F_i|>0$, then the set obtained  by replacing $ E_{i_0} $ by $F_1$ and replacing an empty 
component by $F_2$ will have (strictly) lower energy.
\end{proof}

Next, the crucial tool to pass to the limit $\beta\to0$ and prove our result is the following non-degeneracy estimate:
\begin{proposition}\label{prop:NDG1}
Let $\mathbb E = \{E_i\}_{i=1 ,\dots, N}$ be a  global minimizer of \eqref{eq:minGNR}.
There exists a constant $c$ depending only on $n$ and $s$ such that
\[
|E_i \cap B_r(x_0)|\geq c\min \left\{ m_i, \frac{m_i}{(n\omega_n \gamma_i)^{\frac{n}{1-s}}}, \frac{1}{m_i \beta^n},r^n\right\}
, \qquad \gamma_i =  \alpha\, m_i^{\frac{1-s}{n}}
\]
for every $x_0\in \partial E_i$ and $r>0$.
%\[
%|E_i\cap B_r(x_0)|\geq c r^n \qquad \mbox{ for all $r\leq C \min\{1,\gamma^{-\frac{1}{1-s}}\} m^{1/n}$}
%\]
\end{proposition}
Before proving this proposition, we state and prove the following corollary which allows us to  pass to the limit $\beta\to 0$ and prove Theorem \ref{thm:G}:
\begin{corollary}\label{cor:diam}
Let $\mathbb E = \{E_i\}_{i=1 ,\dots, N}$ be a  global minimizer of \eqref{eq:minGNR}. If  $ \beta \leq \frac{1}{m^{2/n}}$, then 
$$
\mathrm{diam} (E_i) \leq C \max\{1,(n\omega_n\gamma_i)^{\frac{n}{1-s}} \} m_i^{1/n}
$$
for all $i$.
In particular, we can assume that $E_i\subset B_{R_0}$ for some $R_0$ depending on $m$ and $\alpha$ but not on $\beta$.
\end{corollary}

\begin{proof}[Proof of Corollary \ref{cor:diam}]

Without loss of generality, we can assume that there exists $x$, $y$ in $E_i$ such that $(y-x)\cdot e_1 \geq \mathrm{diam} (E_i) -cm_i ^{1/n}$.
Since $E_i$ is indecomposable, we can find $N$ balls $B_{m_i^{1/n}}(x_k)$ with $x_0=x$, $x_N = y$ and $(x_{k+1} -x_k)\cdot e_1 \leq 2m_i ^{1/n}$.
We then have 
\begin{equation}\label{eq:diambd} 
\mathrm{diam} (E_i) \leq C (N+1) m_i ^{1/n} \end{equation}
and Proposition \ref{prop:NDG1} implies that 
$$
|E_i \cap B_{m_i^{1/n}}(x_k)|\geq c\min \left\{1, \frac{1}{(n\omega_n\gamma_i)^{\frac{n}{1-s}}}\right\} m_i,
$$
as long as  $\beta \leq \frac{1}{m_i^{2/n}}$, which is  satisfied for all $i$ if $ \beta \leq \frac{1}{m^{2/n}}$.
We thus have 
$$m_i= |E_i| \geq c \sum_k  |E_i \cap B_{m_i^{1/n}}(x_k)| \geq c N \min \left\{1, \frac{1}{(n\omega_n \gamma_i)^{\frac{n}{1-s}}}\right\} m_i$$
which gives:
$$N\leq C \max\{1,(n\omega_n \gamma_i)^{\frac{n}{1-s}} \} $$
and the result now follows from \eqref{eq:diambd}.

\end{proof}

\begin{proof}[Proof of Proposition \ref{prop:NDG1}]
We denote $E=E_{i_0}$ and $m=m_{i_0}$ and we note that $E$ must minimize $\J_\alpha$ with the constraints $E\subset B_R$ and $|E|=m$.

Given $r>0$
we define the  functions:
$$
A(r)=H^{n-1}(\partial B_r(x_0)\cap E),
\quad S(r)=P(E,B_r(x_0)),
\quad V(r)=|E\cap B_r(x_0)|,
$$
%We know $V(0)=0$ and $V(r)\leq |B_1|r^n$ is non-decreasing in $r$. We want to prove when $r<C_2r_0$,
%\[
%V(r)\geq Cr^n.\\
%\]
and the set
\[
E'=E\setminus B_r(x_0).
\]
Proceeding as in the proof of Lemma \ref{nondegeneracyE}, we then have
\begin{align}
\J_{\alpha,\beta}(E') & = \J_{\alpha,\beta} (E) -S(r) + A(r) +\alpha s(1-s) (P_K(E)-P_K(E')) -\beta \int_{E\cap B_r(x_0)} |x| \, dx\nonumber\\
& \leq \J_{\alpha,\beta} (E) -S(r) + A(r) + \frac 1 2 (S(r)+A(r)) + \frac 1 2 (n\omega_n \alpha)^\frac{1}{1-s} V(r)\nonumber \\
& \leq \J_{\alpha,\beta} (E) -\frac 1 2  S(r) + \frac 3 2  A(r) +  \frac 1 2  \gamma^\frac{1}{1-s} m^{-1/n}V(r)\label{eq:EE'}
\end{align}

Next, we note that for $\lambda>0$, and $E$ minimizer of $\J_{\alpha,\beta}$ with $|E|=m$,
we have
\begin{align*}
\J_{\alpha,\beta}((1+\lambda)^{\frac{1}{n}} E) 
& \leq  (1+\lambda)^\frac{n-1}{n} P(E) - \alpha s(1-s) P_K((1+\lambda)^{\frac{1}{n}} E)  + \beta (1+\lambda)^\frac{n+1} {n}\int_{\R^n} |x| \chi_E\, dx \nonumber\\
& \leq  \J_{\alpha,\beta}(E) 
+ \frac{n-1}{n} \lambda P(E) + \frac{n+1}{n} \lambda \beta \int_{\R^n} |x| \chi_E\, dx \nonumber \\
& \qquad + \alpha s(1-s)(P_K(E)- P_K((1+\lambda)^{\frac{1}{n}} E)  ).
\end{align*}
Using  \eqref{eq:minbound} together with \eqref{eq:PKEL} 
%the fact that 
%$$ \alpha s(1-s) [P_K(E)- P_K\left((1+\lambda)^{\frac 1 n}E\right) ]
%\leq C \alpha \lambda P(E)^s m^{1-s} \leq C \lambda\left[s P(E)+ (1-s) m \right]
%$$
we get:
\begin{align*}
\J_{\alpha,\beta}((1+\lambda)^{\frac{1}{n}} E) 
& \leq  \J_{\alpha,\beta}(E) 
+ \left[ C(1+(n\omega_n \gamma)^\frac{1}{1-s}) m^{\frac{n-1}{n}} + \beta m^\frac{n+1}{n}\right] \lambda 
\end{align*}
and so
$$
\inf_{|F| = (1+\lambda) m} \J_{\alpha,\beta}(F)
\leq 
\inf_{|F| =   m} \J_{\alpha,\beta}(F)+ \left[ C(1+(n\omega_n  \gamma)^\frac{1}{1-s}) m^{\frac{n-1}{n}} + \beta m^\frac{n+1}{n}\right] \lambda 
$$
which also implies
$$
\inf_{|F| = m} \J_{\alpha,\beta}(F)
\leq 
\inf_{|F| =   \frac{m}{1+\lambda} } \J_{\alpha,\beta}(F)+ \left[ C(1+(n\omega_n  \gamma)^\frac{1}{1-s}) m^{\frac{n-1}{n}} + \beta m^\frac{n+1}{n}\right] \lambda 
$$

Choosing $\lambda$ such that $(1+\lambda) |E' | = m$, that is
$$
\lambda:=\frac{V(r)}{m-V(r)}$$
we deduce from \eqref{eq:EE'} that
\begin{align*}
\J_{\alpha,\beta}(E)
& \leq \J_{\alpha,\beta} (E) -\frac 1 2  S(r) + \frac 3 2  A(r) + C  (n\omega_n  \gamma)^\frac{1}{1-s} m^{-1/n}V(r) \\
& \qquad + \left[ C(1+(n\omega_n \gamma)^\frac{1}{1-s}) m^{\frac{n-1}{n}} + \beta m^\frac{n+1}{n}\right] \lambda 
\end{align*}
that is
\begin{align*}
 S(r) & \leq  C  A(r) + C  (n\omega_n \gamma)^\frac{1}{1-s} m^{-1/n}V(r)+ \left[ C(1+(n\omega_n \gamma)^\frac{1}{1-s}) m^{\frac{n-1}{n}} + \beta m^\frac{n+1}{n}\right] \lambda 
\end{align*}

Given $\bar r>0$, we assume that 
$$
V(\bar r)\leq \eps_0 \min\left\{ 1,\frac{1}{(n\omega_n \gamma)^{\frac{n}{1-s}}}, \frac{1}{m^2 \beta^n}\right\} m 
$$
for some $\eps_0<1/2$ to be chosen later. 
Since $r\mapsto V(r)$ is non-decreasing, we then have 
\begin{equation}\label{eq:eps02} 
V( r)\leq \eps_0 \min\left\{ 1,\frac{1}{(n\omega_n  \gamma)^{\frac{n}{1-s}}}, \frac{1}{m^2 \beta^n} \right\} m  \quad\mbox{ for all $r\leq \bar r$,}
\end{equation}
and so in particular 
$$
 \lambda(r)=\frac{V(r)}{m-V(r)} \leq 2 \frac{V(r)}{m}\quad\mbox{ for all $r\leq \bar r$.}
$$

Putting everything together, we deduce:
\begin{equation*}
\begin{aligned}
S(r)
%&\leq 3\mu (Cm^{\frac{n-1}{n}}-S(r)+A(r))+C\mu \alpha^{\frac{1}{1-s}}m\\
%&+\frac{1}{2}(S(r)+A(r))+C_4\alpha^{\frac{1}{1-s}}V(r)\\
&\leq C A(r)+  C  (1+ (n\omega_n  \gamma)^\frac{1}{1-s}) m^{-1/n}V(r) 
+ C \beta m^\frac{1}{n} V(r)\\
&\leq C A(r)+  C\left[   (1+(n\omega_n  \gamma)^\frac{1}{1-s}) (V(r)/m )^{1/n}
+ C \beta (m V(r))^\frac{1}{n} \right] V(r)^{\frac{n-1}{n}}\\
&\leq C A(r)+  C \eps_0^{1/n} V(r)^{\frac{n-1}{n}}
\end{aligned}
\end{equation*} 
In order to conclude, we combine \eqref{eq:SA} with the isoperimetric inequality
\[
S(r)+A(r)\geq \nu_n V(r)^{\frac{n-1}{n}},
\]
and the fact that
\[
V'(r)=A(r) \qquad\mbox{for a.e. $r>0$},
\]
to conclude:
\begin{align*}
\nu_nV(r)^{\frac{n-1}{n}}&\leq C A(r)+C\eps_0  V(r)^\frac{n-1}{n}  \\
&\leq CV'(r)+C\eps_0  V(r)^\frac{n-1}{n}
\end{align*}
%Since $\mu=\frac{V(r)}{m-V(r)}$, then when $r<C_2r_0$ with $C_2=\min\{\frac{C_n}{48C},\frac{1}{2}\}$, we have $V(r)<C_2^nm$ and thus
%\[
%6C\mu m^{\frac{n-1}{n}}\leq\frac{1}{4}C_nV(r)^{\frac{n-1}{n}}.
%\] 
%Also, when $\alpha m^{\frac{1-s}{n}}<\gamma_2$,
%\[
%2C \mu \alpha^{\frac{1}{1-s}}m\leq \frac{1}{12C}\alpha^{\frac{1}{1-s}}m^{\frac{1}{n}}C_nV(r)^{\frac{n-1}{n}}\leq \frac{1}{4}C_nV(r)^{\frac{n-1}{n}}.
%\]
%When $r<C_2r_0$ and $\alpha m^{\frac{1-s}{n}}<\gamma_2$,
%\[
%2C_4\alpha^{\frac{1}{1-s}}V(r)\leq \frac{1}{4}C_nV(r)^{\frac{n-1}{n}}.
%\]
%Thus from \eqref{V'(r)geq} we obtain the following ODE for $V(r)$ when $\gamma<\gamma_2$ and $0<r<C_2r_0$,
Choosing $\eps_0$ small enough, we thus have
$$
V'(r)\geq \frac{\nu_n}{2} V(r)^{\frac{n-1}{n}},\qquad\mbox{for a.e. $0<r<\bar r$}
$$
with $V(0)=0$, which implies
\[
V(\bar r)\geq \left(\frac{\nu_n}{2n}\right)^n \,{\bar r}^n 
\]
and the result follows.
\end{proof}

\section{Proof of Propositions \ref{prop:as} and \ref{prop:JGN}}\label{sec:asJ}
%Asymptotic behavior of $\inf_{E\in \E_m} \J_\alpha(E)$ when $m\to \infty$
%We are considering $J(E)=P(E)-P_s(E)$ with $|E|=m$, which is the case of $\mathcal{E}_{\alpha}$ when $J(x)=|x|^{-(n+s)}$ and $\alpha=1$. Thus we have the existence of minimizers when $m<m_0$ and uniqueness of minimizers when $m<m_1$. In this section, we want to discuss the asymptotic behaviour of $\inf{\{J(E),|E|=m\}}$ when $m$ is large.
%\begin{lemma}
%Let $J(E)=P(E)-P_s(E)$ with $|E|=m$. Then
%\[
%\inf_{|E|=m}{J(E)}=-C_s\mu^{\frac{1}{1-s}}m+o(m)
%\]
%with $\mu>0$ the optimal constant in the Gagliardo-Nirenberg equalities
%\begin{align}\label{GNinequality}
%P_s(E)\leq \mu P(E)^s|E|^{1-s}.
%\end{align}
%\end{lemma}
\begin{proof}[Proof of Proposition \ref{prop:as}]
Using the scaling properties of $\F_{s,\alpha}$, it is enough to prove the result when $\alpha=1$. We denote
$$\F = P-s(1-s )P_s$$
and $\G$ the corresponding generalized functional.
Next we note that we clearly have $\inf_{|F|=m} \F (F) \geq \inf_{|\mathbb F|=m} \G (\mathbb F)$ (since we can take $\mathbb F = \{ F\}$, and the other inequality is obtained by recalling (Theorem \ref{thm:G}) that the minimizer of $\G$ has finitely many bounded component. We can thus construct a sequence $F_k$ by taking those components and sending them to infinity is different direction to get $\F(F_k) \to \inf_{|\mathbb F|=m} \G (\mathbb F)$.

We recall that  $\mu_{n,s}>0$, defined in \eqref{eq:opt}, 
denotes the optimal constant in the interpolation inequality \eqref{eq:GN}.
We thus have
$$
\F (E)=P(E)- s(1-s)P_s(E)\geq P(E)-\mu_{n,s} P(E)^s m^{1-s}.
$$
Since the minimum of the function $g(x)= x - \beta x^s$  is $-(1-s)s^{\frac{s}{1-s}} \beta^{\frac{1}{1-s}}$, which is attained when $x=(s\beta)^{\frac{1}{1-s}}$, we deduce
\begin{equation}\label{JEgeq}
\begin{aligned}
\F (E) \geq -C_s \mu_{n,s}^{\frac{1}{1-s}}m, \qquad C_s=(1-s)s^{\frac{s}{1-s}}>0.
\end{aligned}
\end{equation}
We also note that we have equality if $E$ achieves the infimum in \eqref{eq:opt} and if 
$P(E)=(s \mu_{n,s})^{\frac{1}{1-s}}m$.
\medskip

Since we do not know whether there exists a set $E$ which achieves the optimal  constant in  \eqref{eq:GN}, we consider, given $\eps>0$, a set $E_0$ such that $|E_0|=1$ and 
\begin{align}\label{GNconstant}
P_s(E_0)\geq(\mu_{n,s}-\epsilon)P(E_0)^s
\end{align}
In order to find a set for which there is almost equality in \eqref{JEgeq}, 
we take $k\in \N$ and consider the set 
$E_k$ made of $k$ copies of the rescaled set $(\frac{m}{k})^{1/n}E_0$ (with mass $\frac m k$) positioned in $\R^n$ so that the distance between any two copies is at least $R\gg1$.
The set $E_k$ satisfy $|E_k|=m$ and 
\begin{align}
\G(E_k)&=k\left( \left(\frac{m}{k}\right)^{\frac{n-1}{n}}P(E_0)-\left(\frac{m}{k}\right)^{\frac{n-s}{n}}P_s(E_0) \right)\nonumber \\
&=m^\frac{n-1}{n} P(E_0) \left[ k^{1/n}-k^{s/n} m^\frac{1-s}{n} \frac{P_s(E_0)}{P(E_0)}\right] \nonumber\\
&=m^\frac{n-1}{n} P(E_0) g(k^{1/n}) 
\label{JEk}
\end{align}
where
$$g(x) = x-x^s m^\frac{1-s}{n} \frac{P_s(E_0)}{P(E_0)}.$$
Since we can take $R\to\infty$ to find $\inf \F(E)$, we will discard the $\mathcal O(1/R^{n+s})$ below (this amount to consider the set $E_k$ in which the $k$ copies of $(\frac{m}{k})^{1/n}E_0$  have been sent to infinity in different directions).
Furthermore, we note that the function $g(x)$ is minimum when 
$x=\bar x = \left(\frac{sP_s(E_0)}{P(E_0)}\right)^{\frac{1}{1-s}}m^\frac{1}{n}$ which correspond to 
$$k=\bar k :=\left(\frac{sP_s(E_0)}{P(E_0)}\right)^{\frac{n}{1-s}}m$$
However $\bar{k}$ is not necessarily an integer, so we take
$k = [\bar k]$ (the integer part of $\bar k$) and denote $\delta = \{k\} :=k-[k]\in[0,1)$ the fractional part of $\bar k$.
The convexity of $g$ implies
\begin{align*}
 g(k^{1/n}) 
 & \leq g(\bar k^{1/n})+g'(k^{1/n}) (k^{1/n}-\bar k^{1/n})\\
 & \leq -C_s m^{\frac{1}{n}} \left(\frac{P_s(E_0)}{P(E_0)}\right)^{\frac{1}{1-s}}+\left( 1-  s\left(\frac m k\right)^{\frac{1-s}{n}} \frac{P_s(E_0)}{P(E_0)} \right) (k^{1/n}-\bar k^{1/n})\\
 & \leq -C_s m^{\frac{1}{n}} \left(\frac{P_s(E_0)}{P(E_0)}\right)^{\frac{1}{1-s}}+\left( 1-  \left(\frac {\bar k} k\right)^{\frac{1-s}{n}}  \right) (k^{1/n}-\bar k^{1/n})\\
 & \leq -C_s m^{\frac{1}{n}} \left(\frac{P_s(E_0)}{P(E_0)}\right)^{\frac{1}{1-s}}+
C k^{\frac 1 n -2}\delta^2 
 \end{align*}
 where the constant $C$ depends only on $n$ and $s$. We deduce
\begin{align}
\G(E_k)
&\leq   -C_s m \left(\frac{P_s(E_0)}{P(E_0)^s}\right)^{\frac{1}{1-s}}    +
C m^\frac{n-1}{n} P(E_0) k^{\frac 1 n -2}\delta^2 
  \nonumber \\
&\leq   -C_s\left(\mu_{n,s}-\eps\right)^{\frac{1}{1-s}}  m  +
C m \left(\frac{\delta}{k}\right)^2  \label{eq:minFE}
\end{align}
Since $\delta\in[0,1)$ we see that there exists $m_\eps$ (importantly, $m_\eps$ depends on $P(E_0)$ which might go to infinity as $\eps\to0$) such that if $m\geq m_\eps$ then 
\begin{align*}
\G(E_k)
 &\leq   -C_s  \mu_{n,s} ^{\frac{1}{1-s}} m   +
C\eps  m  
\end{align*}
and so
$$
\inf _{|E|=m} \G(E) \leq   -C_s  \mu_{n,s}  ^{\frac{1}{1-s}} m   +
C\eps  m \qquad \mbox{ for all } m\geq m_\eps.$$
The result follows.
\end{proof}

\begin{proof}[Proof of Proposition \ref{prop:JGN}]
\noindent{ (i)} If there exists a set with finite perimeter $E_0$  such that $|E_0|=1$ and 
$$ P_s(E_0) = \mu_{n,s} P(E_0)^s,$$
then we can take $\eps=0$ in the argument above, and \eqref{eq:minFE} implies
\begin{align*}
\G(\mathbb E_k) & \leq -C_s\mu_{n,s}^{\frac{1}{1-s}}  m  +
C m \left(\frac{\delta}{k}\right)^2 .
\end{align*}
For all $m$ such that $k=\left(\frac{sP_s(E_0)}{P(E_0)}\right)^{\frac{n}{1-s}}m\in \N$, we have $\delta=0$ and so 
\begin{align*}
\G(\mathbb  E_k) & \leq -C_s\mu_{n,s}^{\frac{1}{1-s}}  m  .
\end{align*}
Together with \eqref{JEgeq}, it shows that $\mathbb  E_k$ is a minimizer for $\G$.

This shows that whenever 
$$m\in \left\{ k \left(\frac{P(E_0)}{sP_s(E_0)}\right)^{\frac{n}{1-s}}\, ;\,  k\in \N\right\}$$
then a minimizing sequence for $\F$ with volume constraint $|E|=m$ can be constructed by taking 
$k$ copies  of the set $(\frac{m}{k})^{1/n}E_0$ and sending them to infinity in different directions.

\medskip

\noindent{ (ii)} Assume there exists a sequence $\{m_k\}_{k=1}^{\infty}$ with $m_k \to \infty$ such that there exist bounded minimizers $E_{k}$ of $\F_{s,\alpha}(F)$ with $|F|=m_k$. Using the first part of Proposition \ref{prop:as}, we have
\begin{equation*}
\F_{s,\alpha}(E_k)=\inf_{|F|=m_k}{\F_{s,\alpha}(F)}=-(1-s)s^{\frac{s}{1-s}} (\alpha\mu_{n,s})^{\frac{1}{1-s}}m_k+o(m_k).
\end{equation*}
Then the set $F_k=m_k^{-\frac{1}{n}}E_k$ with $|F_k|=1$ satisfies
\begin{equation}\label{eq:GNsequence}
\lim_{k\to \infty}{P(F_k)-\alpha s(1-s)m_k^{\frac{1-s}{n}}P_s(F_k)}=-(1-s)s^{\frac{s}{1-s}} (\alpha\mu_{n,s})^{\frac{1}{1-s}}m_k^{1/n}.
\end{equation}
Using the same argument when proving \eqref{JEgeq}, we have
\[
P(F_k)-\alpha s(1-s)m_k^{\frac{1-s}{n}}P_s(F_k)\geq P(F_k)-\alpha \mu_{n,s}m_k^{\frac{1-s}{n}}P^s(F_k)\geq -(1-s)s^{\frac{s}{1-s}}(\alpha \mu_{n,s})^{\frac{1}{1-s}}m_k^{\frac{1}{n}},
\]
which implies $\{F_k\}$ is a minimizing sequence for the optimal constant in Galgliardo-Nirenberg inequality \eqref{eq:opt}.
From \eqref{eq:GNsequence} we have
\[
\frac{c_n}{s(1-s)} P^s(F_k)\geq P_s(F_k)\geq s^{\frac{2s-1}{1-s}}\alpha^{\frac{s}{1-s}}\mu_{n,s}^{\frac{1}{1-s}}m_k^{s/n},
\]
which implies $P(F_k)\geq Cm_k^{\frac{1}{n}}\to \infty$ as $k \to \infty$. 
\end{proof}

\appendix

\section{Proof of Proposition \ref{prop:GN}}\label{app:GN} 
First, we write
\begin{align*}
P_K(E) &  = \int_{\R^n}  K(z) \int_{\R^n} |\chi_E(x)-\chi_E(x+z)|dz dx \\
&  = \int_{|z|\geq R}  K(z) \int_{\R^n} |\chi_E(x)-\chi_E(x+z)|dx dz+\int_{|z|\leq R}  K(z) \int_{\R^n} |\chi_E(x)-\chi_E(x+z)|dx\, dz
\end{align*}
Using the bound
$$  \int_{\R^n} |\chi_E(x)-\chi_E(x+z)| dx \leq 2|E|$$
in the first integral and 
\begin{equation}\label{eq:BVbd}
 \int_{\R^n} |\chi_E(x)-\chi_E(x+z)|dx  \leq  |z| \int_{\R^n} |D\chi_E|   
 \end{equation}
in the second integral, we get:
\begin{align*}
P_K(E) &\leq 2|E|  \int_{|z|\geq R}  K(z) dz+ P(E) \int_{|z|\leq R} |z|  K(z)  \,  dz\\
&\leq n \omega_n \left( 2|E|  \frac{R^{-s}}{s}+ P(E) \frac{R^{1-s}}{1-s} \right).
\end{align*}
Optimizing with respect to $R$ by taking $R=\frac{2|E|}{P(E)}$ yields the result.

\section{Proof of Lemma \ref{lem:CVPK}} \label{sec:CVPK}
We write
\begin{align*}
| P_K(F_k) - P_K(E) |
&\leq 
\left| \int_{\R^n}  \int_{|x-y|\geq \eta} K(x-y)|\chi_{F_k}(x)-\chi_{F_k}(y)|dx\, dy -\right.\\
& \qquad\qquad  \left. \int_{\R^n} \int_{|x-y|\geq \eta} K(x-y)|\chi_E(x)-\chi_E(y)|dx\, dy \right| \\
& \quad + \int_{\R^n} \int_{|x-y|\leq \eta} K(x-y)|\chi_{F_k}(x)-\chi_{F_k}(y)|dx\, dy\\
& \quad + \int_{\R^n} \int_{|x-y|\leq \eta} K(x-y)|\chi_{E}(x)-\chi_{E}(y)| dx\, dy
\end{align*}
The strong convergence in $L^1(\R^n)$ implies that the first term converges to $0$ as $k\to\infty$.
Using \eqref{eq:BVbd}, we deduce
$$
\limsup_{k\to\infty}  | P_K(F_k) - P_K(E) |\leq (P(F_k)+P(E)) \int_{\R^n} \int_{|x-y|\leq \eta} K(x-y)|x-y|dx\, dy \leq C \eta^{1-s}.
$$
Since this holds for all $\eta>0$, the result follows.

\section{Proof of Lemma \ref{lem:psi}}\label{sec:psi}
It is equivalent to prove
\begin{equation}\label{eq:phi}
\left|\int_{\partial B_1}{(\psi(1+\lambda)-\psi(1))d\mathcal{H}^{n-1}(x)}-P(B_1)\beta \lambda\right|\leq \frac{C}{s(1-s)}P(B_1)r_0^{n-s}\lambda^2.
\end{equation}
Using \eqref{P_J-psi}, we write
\begin{align*}
\int_{\partial B_1}{\psi(1+\lambda)d\mathcal{H}^{n-1}(x)}&=P_K(B_{(1+\lambda)r_0})\\
&=\frac{1}{2}\int_{\mathbb{R}^n}{\int_{\mathbb{R}^n}{|\chi_{B_{(1+\lambda)r_0}}(x)-\chi_{B_{(1+\lambda)r_0}}(x+z)|K(z)dz}dx}\\
&=\frac{1}{2}\int_{\mathbb{R}^n}{\int_{\mathbb{R}^n}{|\chi_{B_{r_0}}(x)-\chi_{B_{r_0}}(x+z)|(1+\lambda)^{2n}K((1+\lambda)z)dz}dx}\\
&=\frac{1}{2}\int_{\mathbb{R}^n}{\int_{\mathbb{R}^n}{\left|\chi_{B_{r_0}}(x)-\chi_{B_{r_0}}(x+\frac{z}{1+\lambda})\right|(1+\lambda)^{n}K(z)dz}dx}\\
&=\frac{1}{2}(1+\lambda)^{n}\int_{\mathbb{R}^n}{\phi\left(\frac{z}{1+\lambda}\right)K(z)dz},
\end{align*}
where we defined
\[
\phi(|z|):=\phi(z)=\int_{\mathbb{R}^n}{|\chi_{B_{r_0}}(x)-\chi_{B_{r_0}}(x+z)|dx}
\]
which is a radially symmetric function. 

Setting
\[
F(\lambda):=(1+\lambda)^{n}\int_{\mathbb{R}^n}{\phi\left(\frac{z}{1+\lambda}\right) K(z)dz},
\]
which satisfies in particular,
\[
F(0)=P_K(B_{r_0})=\int_{\partial B_1}{\psi(1)d\mathcal{H}^{n-1}(x)},
\]
we see that \eqref{eq:phi} is equivalent to showing that there exists $\tilde{\beta}=P(B_1)\beta\in \mathbb{R}$ such that
\[
|F(\lambda)-F(0)-\tilde{\beta}\lambda|<\frac{C}{s(1-s)}r_0^{n-s}\lambda^2
\]
for all $|\lambda|<1/2$.  Clearly, this holds if we take $\tilde{\beta}=F'(0)$ and prove 
\[
|\tilde{\beta}|\leq \frac{C}{s(1-s)}r_0^{n-s},
\]
and 
\[
\sup_{|\lambda|<1/2}{|F''(\lambda)|}\leq \frac{C}{s(1-s)}r_0^{n-s}.
\]

To prove these, we note that for all $0\leq t\leq 2r_0$, we have:
\begin{align*}
\phi(t)& =2|B_{r_0}|-4\int_{t/2}^{r_0}{\omega_{n-1}(r_0^2-\tau^2)^{\frac{n-1}{2}}d\tau},\\
\phi'(t) & =4\omega_{n-1}\left(r_0^2-\frac{t^2}{4}\right)^{\frac{n-1}{2}},\\
\phi''(t) & =-(n-1)\omega_{n-1}t\left(r_0^2-\frac{t^2}{4}\right)^{\frac{n-3}{2}},
\end{align*}
with $\omega_{n-1}$ the $n-1$ dimensional unit ball volume,
while when $t>2r_0$, we find
\[
\phi(t)=2|B_{r_0}|, \quad \phi'(t)=0, \quad \phi''(t)=0. 
\]
We thus have:
\begin{align*}
F'(\lambda)&=n(1+\lambda)^{n-1}\int_{\mathbb{R}^n}{\phi\left(\frac{z}{1+\lambda}\right)K(z)dz}\\
& \qquad -(1+\lambda)^{n-2}\int_{\mathbb{R}^n}{\phi'\left(\frac{|z|}{1+\lambda}\right)|z|K(z)dz}
\end{align*}
and therefore
\begin{align*}
\tilde{\beta}:=F'(0)&=n\int_{\mathbb{R}^n}{\phi(z)K(z)dz}-\int_{\mathbb{R}^n}{\phi'(|z|)|z|K(z)dz}\\
&=nP_K(B_{r_0})-4P(B_1)\omega_{n-1}\int_0^{2r_0}{\left(r_0^2-\frac{t^2}{4}\right)^{\frac{n-1}{2}}t^n K(t)dt}\\
&=nP_K(B_{r_0})-4P(B_1)\omega_{n-1}r_0^{n-s}\int_0^{2}{\left(1-\frac{t^2}{4}\right)^{\frac{n-1}{2}}t^n K(t)dt}
\end{align*}
and the integral is bounded since near $t=0$, $\int_0^2{t^n K(t)dt}\leq \frac{C}{1-s}$ and $P_K(B_{r_0})\leq \frac{C}{s(1-s)}r_0^{n-s}$. Thus we have 
\begin{align*}
|\tilde{\beta}|\leq \frac{C}{s(1-s)}r_0^{n-s}
\end{align*}
which implies
\begin{align}\label{betabound}
|\beta|=|\frac{\tilde{\beta}}{P(B_1)}|\leq \frac{C}{s(1-s)}r_0^{n-s}.
\end{align}

Next, we prove that
\[
\sup_{|\lambda|<1/2}{|F''(\lambda)|}\leq \frac{C}{s(1-s)}r_0^{n-s}.
\]
We can see
\begin{align*}
F''(\lambda)&=n(n-1)(1+\lambda)^{n-2}\int_{\mathbb{R}^n}{\phi\left(\frac{z}{1+\lambda}\right)K(z)dz}\\
&\quad -(2n+2)(1+\lambda)^{n-3}\int_{\mathbb{R}^n}{\phi'\left(\frac{|z|}{1+\lambda}\right)|z|K(z)dz}\\
&\quad +(1+\lambda)^4\int_{\mathbb{R}^n}{\phi''\left(\frac{|z|}{1+\lambda}\right)|z|^2K(z)dz}\\
&=J_1+J_2+J_3.
\end{align*}
The first integral
\[
|J_1|\leq \frac{C}{(1+\lambda)^2}P_K(B_{(1+\lambda)r_0})\leq CP_s(B_{(1+\lambda)r_0})\leq \frac{C}{s(1-s)}r_0^{n-s}.
\]
The second integral
\begin{align*}
|J_2|&\leq C\int_{\mathbb{R}^n}{\phi'\left(\frac{|z|}{1+\lambda}\right)|z|K(z)dz}\\
&=CP(B_1)\int_{0}^{2r_0(1+\lambda)}{t^{n-1}\phi'\left(\frac{t}{1+\lambda}\right)t K(t)dt}\\
&\leq C\int_{0}^{2r_0(1+\lambda)}{t^{n-1}\left(r_0^2-\frac{t^2}{4(1+\lambda)^2}\right)^{\frac{n-1}{2}}t\frac{1}{t^{n+s}}dt}\\
&\leq Cr_0^{n-s}\int_{0}^{2(1+\lambda)}{t^{-s}\left(1-\frac{t^2}{4(1+\lambda)^2}\right)^{\frac{n-1}{2}}dt}\\
&\leq \frac{C}{1-s}r_0^{n-s}\\
\end{align*}
since the integral is bounded by $\int_0^3{t^{-s}dt}\leq \frac{C}{1-s}$.
The last integral
\begin{align*}
|J_3|&= CP(B_1)\int_{0}^{2r_0(1+\lambda)}{\left|\phi''\left(\frac{t}{1+\lambda}\right)\right| t^2 K(t)t^{n-1}dt}\\
&\leq C\int_{0}^{2r_0(1+\lambda)}{t\left(r_0^2-\frac{t^2}{4(1+\lambda)^2}\right)^{\frac{n-3}{2}}t^2t^{-n-s}t^{n-1}dt}\\
&=Cr_0^{n-s}\int_0^{2(1+\lambda)}{\left(1-\frac{t^2}{4(1+\lambda)^2}\right)^{\frac{n-3}{2}}t^{2-s}dt}\\
&=Cr_0^{n-s}\int_0^{2(1+\lambda)}{\left(1-\frac{t}{2(1+\lambda)}\right)^{\frac{n-3}{2}}\left(1+\frac{t}{2(1+\lambda)}\right)^{\frac{n-3}{2}}t^{2-s}dt}\\
&\leq Cr_0^{n-s}
\end{align*}
since the singularity when $n=2$ near $t=2(1+\lambda)$ is bounded by $(1-\frac{t}{2(1+\lambda)})^{-1/2}$ which is integrable, and when $n\geq 3$, there is no singularity. Thus, we prove
\[
\sup_{|\lambda|<1/2}{|F''(\lambda)|}\leq \frac{C}{s(1-s)}r_0^{n-s}.
\]
Therefore, with our choice of $\tilde{\beta}=F'(0)$,
\[
|F(\lambda)-F(0)-\tilde{\beta}\lambda|<\frac{C}{s(1-s)}r_0^{n-s}\lambda^2.
\]

\bibliographystyle{siam}
\bibliography{minimizer}

\end{document}